\documentclass[12pt]{article}
\usepackage{amsmath, amsthm}\usepackage{enumerate}\usepackage{amssymb}\usepackage{bbold}
\newtheorem{theorem}{Theorem}
\newtheorem{example}{Example}
\newtheorem{remark}{Remark}
\newtheorem{proposition}{Proposition}[section]
\newtheorem{lemma}{Lemma}[section]
\newtheorem{corollary}{Corollary}[section]
\newtheorem{definition}{Definition}[section]

\DeclareMathOperator{\convr}{\xrightarrow[]{\mathbb{r}}}

\DeclareMathOperator{\convo}{\xrightarrow[]{\mathbb{o}}}
\DeclareMathOperator{\convuo}{\xrightarrow[]{\mathbb{uo}}}

\DeclareMathOperator{\convc}{\xrightarrow[]{\mathbb{c}}}
\DeclareMathOperator{\convac}{\xrightarrow[]{\mathbb{ac}}}
\DeclareMathOperator{\convfc}{\xrightarrow[]{\mathbb{fc}}}
\DeclareMathOperator{\convsc}{\xrightarrow[]{\mathbb{sc}}}
\DeclareMathOperator{\convsr}{\xrightarrow[]{\mathbb{sr}}}
\DeclareMathOperator{\convuIc}{\xrightarrow[]{u_I\mathbb{c}}}
\DeclareMathOperator{\convuIuIc}{\xrightarrow[]{u_I(u_I\mathbb{c})}}
\DeclareMathOperator{\convuc}{\xrightarrow[]{u\mathbb{c}}}
\DeclareMathOperator{\convmc}{\xrightarrow[]{\mathbb{mc}}}

\DeclareMathOperator{\convmo}{\xrightarrow[]{\mathbb{mo}}}
\DeclareMathOperator{\convw}{\xrightarrow[]{\mathbb{w}}}

\DeclareMathOperator{\convcone}{\xrightarrow[]{\mathbb{c}_1}}
\DeclareMathOperator{\convctwo}{\xrightarrow[]{\mathbb{c}_2}}
\DeclareMathOperator{\convtau}{\stackrel{{\text{\boldmath{$\tau$}}}}{\rightarrow}}
\DeclareMathOperator{\convtaum}{\stackrel{{\text{\boldmath{$\tau_m$}}}}{\rightarrow}}
\DeclareMathOperator{\convmtau}{\stackrel{{\text{\boldmath{$m\tau$}}}}{\rightarrow}}

\begin{document}

\title{Full Lattice Convergence on Riesz Spaces}\maketitle\author{\centering{{Abdullah Ayd{\i}n$^{1}$, Eduard Emelyanov$^{2}$, 
Svetlana Gorokhova $^{3}$\\ \small $1$ Department of Mathematics, Mu\c{s} Alparslan University, Mu\c{s}, Turkey\\ \small $2$ Department of Mathematics, Middle East Technical University, Ankara, Turkey\\ 
\small $3$ Southern Mathematical Institute of the Russian Academy of Sciences, Vladikavkaz, Russia\\ 
\abstract
{The full lattice convergence on a locally solid Riesz space is an abstraction of the topological, order, and  relatively uniform convergences.
We investigate four modifications of a full convergence $\mathbb{c}$ on a Riesz space. The first one produces a sequential convergence $\mathbb{sc}$.
The second makes an absolute $\mathbb{c}$-convergence and generalizes the absolute weak convergence. The third modification makes an unbounded $\mathbb{c}$-convergence 
and generalizes various unbounded convergences recently studied in the literature. The last one is applicable whenever $\mathbb{c}$ is a full 
convergence on a commutative $l$-algebra and produces  the multiplicative modification $\mathbb{mc}$ of $\mathbb{c}$. We study general properties 
of full lattice convergence with emphasis on universally complete Riesz spaces and on Archimedean $f$-algebras.
The technique and results in this paper unify and extend those which were developed and obtained in recent literature on unbounded convergences. }\\
\vspace{2mm}

\vspace{2mm}


\section{Introduction and Preliminaries}

This paper deals with real Riesz spaces and various convergences on them. There are several approaches 
to formalization of the notion of convergence (see, e.g., the approach based on the filter convergence \cite{BB,W1} and references therein). 
We prefer an approach based on the following definition. 

\begin{definition}\label{convergence}
Let $S$ be a set. A {\em class $\mathbb{c}$ of pairs} $(C,c)$, where $C$ is a net in $S$ and $c\in S$, 
is called a {\em convergence on $S$}, if 
\begin{enumerate}
\item[$(a)$] \ for any constant net $B\equiv b$ in $S$, $(B,b)\in\mathbb{c}$$;$ 
\item[$(b)$] \ the class $\mathbb{c}$ is closed under passing to subnets, i.e. 
if $(C,c)\in\mathbb{c}$ and $B$ is a subnet of $C$ then $(B,c)\in\mathbb{c}$$;$
\item[$(c)$] \ $((c_\alpha)_{\alpha\in A},c)\in\mathbb{c}$ whenever a {\em tail of the net $(c_\alpha)_{\alpha\in A}$
converges to $c$}, i.e. $((c_\alpha)_{\alpha\in\{\xi\in A: \xi\ge\alpha_0\}},c)\in\mathbb{c}$ for some $\alpha_0\in A$.
\end{enumerate}
\noindent
A convergence $\mathbb{c}$ on a set $S$ is said to be$:$ 
\begin{enumerate}
\item[$(d)$] \ $T_1$ if $[((x_\alpha)_{\alpha\in A}, x)\in\mathbb{c}, \ ((x_\alpha)_{\alpha\in A}, y)\in\mathbb{c}] \ \Longrightarrow \ x=y$,
we write in this case $\mathbb{c}\in T_1$$;$
\item[$(e)$] \ {\em topological} if $\mathbb{c}$ is a convergence with respect to some topology on $S$.
\end{enumerate}
\end{definition}

One of the most common examples of convergence is the usual topological convergence on a topological space.  It is worth noticing that, for two topologies
$\tau_1$ and $\tau_2$ on $S$,  ${\text{\boldmath{$\tau$}}}_1 \subseteq {\text{\boldmath{$\tau$}}}_2$ iff $\tau_2\subseteq\tau_1$, where
we denote by ${\text{\boldmath{$\tau$}}}_1$ and ${\text{\boldmath{$\tau$}}}_2$ the convergences  on $S$ corresponding to topologies $\tau_1$ and $\tau_2$. 
Various non-topological convergences are widely presented through the present paper. Often we shall use the notation $C\convc c$ instead of $(C,c)\in\mathbb{c}$. 
We begin with a ``sequential" modification of convergences.

\begin{definition}\label{sequential c-convergence} 
Let $\mathbb{c}$ be a convergence on a set $S$. 
\begin{enumerate}
\item[$(a)$] \
A net $(x_{\alpha})_{\alpha\in A}$ in $S$ is said to be {\em $\mathbb{sc}$-convergent to $x\in X$} whenever for any subnet $(x_{\alpha_\beta})_{\beta\in B}$ of $(x_{\alpha})_{\alpha\in A}$ 
there exists a sequence $\beta_n$ of elements of $B$ such that $(x_{\alpha_{\beta_n}})_{n\in\mathbb{N}}\convc x$.
\item[$(b)$] \ The convergence $\mathbb{c}$ is called {\em sequential} if $\mathbb{sc}=\mathbb{c}$. 
\end{enumerate}
\end{definition}

It follows immediately from Definition \ref{sequential c-convergence} that $\mathbb{sc}$ is a convergence on $S$ for every convergence $\mathbb{c}$ on $S$,
and hence $\mathbb{ssc}=\mathbb{sc}$. Therefore the convergence $\mathbb{sc}$ is sequential  for any convergence $\mathbb{c}$ on $S$.

\begin{proposition}\label{sequential topological convergence}
Let $\mathbb{c}$ be the topological convergence on a topological space $(T,\tau)$.  The following are equivalent$:$
\begin{enumerate}
\item[$(i)$] \
the convergence $\mathbb{c}$ is sequential$;$
\item[$(ii)$] \
the topology $\tau$ on $T$ is sequential, that is$:$ for every $S\subseteq T$ and for every $x\in\text{cl}_{\tau}(S)$
there exists a sequence $x_n$ in $S$ such that $x_n\convc x$.
\end{enumerate}
\end{proposition}

\begin{proof}
$(i)\Longrightarrow(ii)$ \  Let $x\in\text{cl}_{\tau}(S)\subseteq T$. Then there exists a net $s_\alpha\in S$ which is 
convergent to $x$ with respect to the topology $\tau$, that is $s_\alpha\convc x$. Since the convergence $\mathbb{c}$ is sequential, 
$s_\alpha\convsc x$. Hence there exists a sequence $\alpha_n$ such that $S\ni s_{\alpha_n}\convc x$.
Thus the topology $\tau$ is sequential. 

$(ii)\Longrightarrow(i)$ \ Let the topology $\tau$ be sequential. Take a net $(x_{\alpha})_{\alpha\in A}$ in $T$. 

First suppose that $(x_{\alpha_\beta})_{\beta\in B}$ is a subnet of $(x_{\alpha})_{\alpha\in A}$ and $(x_{\alpha})_{\alpha\in A}\convc x$.
By Definition \ref{convergence}$(b)$, $(x_{\alpha_\beta})_{\beta\in B}\convc x$ and hence $x\in\text{cl}_{\tau}(\{x_{\alpha_\beta}:\beta\in B\})\subseteq T$.  
Since the topology $\tau$ is sequential, there exists a sequence $x_{\alpha_{\beta_n}}$ in the set $\{x_{\alpha_\beta}:\beta\in B\}$ such that $x_{\alpha_{\beta_n}}\convc x$. 
Therefore  $(x_{\alpha})_{\alpha\in A}\convsc x$.

Suppose now $(x_{\alpha})_{\alpha\in A}\convsc x$. Assume on the contrary that $(x_{\alpha})_{\alpha\in A} \stackrel{\mathbb{c}}{\nrightarrow} x$. 
Then there is $U\in\tau(x)$ and a subnet $(x_{\alpha_\beta})_{\beta\in B}$ of $(x_{\alpha})_{\alpha\in A}$ such that
$x_{\alpha_\beta}\not\in U$ for all $\beta\in B$. Since $(x_{\alpha})_{\alpha\in A}\convsc x$, there exists 
a sequence $\beta_n$ in $B$ such that $x_{\alpha_{\beta_n}}\convc x$ and hence $x\in\text{cl}_{\tau}(\{x_{\alpha_{\beta_n}}:n\in\mathbb{N}\})$.
Thus, there exists $n_0$ such that $x_{\alpha_{\beta_n}}\in U$ for all $n\ge n_0$. This is impossible since $x_{\alpha_\beta}\not\in U$ for all $\beta\in B$.
The obtained contradiction shows that $(x_{\alpha})_{\alpha\in A}\convc x$.

Therefore $\mathbb{sc}=\mathbb{c}$ as required.
\end{proof}

In particular, the topological convergence on any metric space is sequential. By Proposition \ref{r convergence is sequential},
the relatively uniform convergence is sequential on any Riesz space. By \cite[Rem.2.6]{DEM3}, the relatively uniform convergence 
is not topological on any infinite dimensional complete metrizable locally solid Lebesgue Riesz space.
Example \ref{sequential o} provides another non-trivial sequential convergence which is not topological.
We do not know any example of a topological space $(T,\tau)$ such that the sequential convergence  
${\text{\boldmath{$s\tau$}}}$ is not topological.  

The following two definitions are taken from \cite{DE}.

\begin{definition}\label{closed and compact}
Let $\mathbb{c}$ be a convergence on a set $S$. A subset $A$ of $S$  is called$:$
\begin{enumerate}
\item[$(a)$] \
{\em $\mathbb{c}$-closed} in $S$ if $A\ni a_\alpha\convc x\in S \Longrightarrow x\in A$$;$
\item[$(b)$] \
{\em $\mathbb{c}$-compact} if any net $a_\alpha$ in $A$ possesses a subnet $a_{\alpha_\beta}$ such that $a_{\alpha_\beta}\convc a$ for 
some $a\in A$.
\end{enumerate}
\end{definition}

\begin{definition}\label{cont mapping} 
Let $S_i$ be sets with convergences $\mathbb{c}_i$, where $i=1,2$.
A mapping $f:S_1\to S_2$ is called {\em $\mathbb{c}_1\mathbb{c}_2$-continuous} if, for every net $x_\alpha$ in $S_1$,
$$
   x_\alpha\convcone x \ \ \Longrightarrow \ \  f(x_\alpha)\convctwo f(x). 
$$
If $\mathbb{c}$ is a convergence on a set $S$ we call a $\mathbb{c}\mathbb{c}$-continuous
mapping $f:S\to S$ just {\em $\mathbb{c}$-continuous}.
\end{definition}

Obviously, for convergences $\mathbb{c}_1$ and $\mathbb{c}_2$ on a set $S$,
the identity mapping in $S$ is $\mathbb{c}_1\mathbb{c}_2$-continuous iff $\mathbb{c}_1\subseteq\mathbb{c}_2$.

\begin{definition}\label{just product convergence} 
Let $\{S_i\}_{i\in I}$ be a family of sets with convergences $\mathbb{c}_i$. 
The product convergence $\mathbb{b}=\prod_{i\in I}\mathbb{c}_i$ on $\prod_{i\in I}S_i$ is defined by
$$
  ((\bar{x}_\alpha)_{\alpha\in A},\bar{x})\in\mathbb{b} \ \ \text{if} \ \ 
  (\forall i\in I)(((x_\alpha)_i)_{\alpha\in A},x_i)\in\mathbb{c}_i,
$$
where $\bar{x}=(x_i)_i\in\prod_{i\in I}S_i$ and $\bar{x}_\alpha=(x_\alpha)_i\in\prod_{i\in I}S_i$ for all $\alpha\in A$.
\end{definition}

\begin{remark}\label{about product convergence}
The following observations are immediate$:$
\begin{enumerate}
\item[$(i)$] \ $\mathbb{b}$ satisfies conditions $(a)$, $(b)$, and $(c)$ of Definition $\ref{convergence}$$;$
\item[$(ii)$] \ the product convergence $\mathbb{b}$ is $T_1$ iff $\mathbb{c}_i\in T_1$ for all $i\in I$$;$
\item[$(iii)$] \ if $I$ is finite and $S_i$ is $\mathbb{c}_i$-compact for every  $i\in I$, then $\prod_{i\in I}S_i$ is $\mathbb{b}$-compact.
\end{enumerate}
\end{remark}

\begin{definition}\label{linear convergence}
A convergence $\mathbb{c}$ on a real vector space $X$ is called {\em linear}, whenever$:$
\begin{enumerate}
\item[$(a)$] \ 
the addition in $X$ is a $(\mathbb{c}\times\mathbb{c})\mathbb{c}$-continuous mapping $X\times X\to X$,
where $\mathbb{c}\times\mathbb{c}$ is the product convergence on $X\times X$$;$
\item[$(b)$] \ 
the scalar multiplication in $X$ is a $({\text{\boldmath{$\tau$}}}\times\mathbb{c})\mathbb{c}$-continuous mapping $\mathbb{R}\times{X}\to X$,
where $\mathbb{R}$ is equipped with the standard topological convergence ${\text{\boldmath{$\tau$}}}$ and ${\text{\boldmath{$\tau$}}}\times\mathbb{c}$ is the product 
convergence on $\mathbb{R}\times X$.
\end{enumerate}
We say that  the convergence $\mathbb{c}$ is  {\em additive} if the addition in $X$ is a $(\mathbb{c}\times\mathbb{c})\mathbb{c}$-continuous mapping $X\times{X}\to X$
and the additive inverse $x\to -x$ in $X$ is a $\mathbb{c}\mathbb{c}$-continuous mapping $X\to X$$;$ that is
\begin{enumerate}
\item[$(c)$] \ 
$x_\alpha \convc x$ and $y_\beta \convc y$ implies $(x_\alpha \pm y_\beta)_{(\alpha,\beta)} \convc x\pm y$.
\end{enumerate}
\end{definition}

Each linear convergence is additive since $(a)$ together with $(b)$ clearly imply $(c)$. Proposition \ref{o-convergence is full lattice and T_1} below
shows that not every additive convergence is linear. 
We do not know of any example of a linear convergence $\mathbb{c}$ on a vector space $X$ such that $\mathbb{sc}$
is not linear. As, in many cases, our approach to convergence is a repetition of well known reasoning, we often skip certain points 
and reproduce standard reasoning in details, where such a repetition seems to be convenient for the reader. 

\begin{lemma}\label{T_1 for linear conv}
An additive convergence $\mathbb{c}$ on a vector space $X$ is $T_1$ iff,  for any constant net $A\equiv a$ in $X$ and any $b\ne a$ in $X$, $(A,b)\notin\mathbb{c}$.

Furthermore, if $\mathbb{c}_1\subseteq\mathbb{c}_2$ for a convergence $\mathbb{c}_1$ and a $T_1$-convergence $\mathbb{c}_2$ on $X$ then $\mathbb{c}_1$ is also $T_1$.
\end{lemma}

\begin{proof}
The necessity is obvious. For sufficiency, assume that $x_\alpha\convc x$ and $x_\alpha\convc y$ in $X$.  The additivity of  $\mathbb{c}$ implies
$0 = (x_\alpha - x_\alpha)  \convc (x -y)$. Since $0  \convc  0$,  the assumption implies $x-y = 0$, and hence $x = y$, as required.

The remaining  part of the lemma is trivial.
\end{proof}

\begin{definition}\label{c-Cauchy}
Let $\mathbb{c}$ be a convergence on a vector space $X$. A net $(a_\gamma)_{\gamma\in\Gamma}$ in $X$
is said to be {\em $\mathbb{c}$-Cauchy} if $(a_{\gamma_1}-a_{\gamma_2})_{(\gamma_1,\gamma_2)\in\Gamma\times\Gamma}\convc 0$.
The vector space $X$ is said to be {\em $\mathbb{c}$-complete} if, for every $\mathbb{c}$-Cauchy net $(a_\gamma)_{\gamma\in\Gamma}$ in $X$,
there exists $a\in X$ such that $(a_\gamma)_{\gamma\in\Gamma}\convc a$. 
\end{definition}

The next two convergences correspond to the {\em anti-discrete} and {\em discrete} topology.

\begin{example}\label{c_a and c_d} 
Let $X$ be a vector space.
\begin{enumerate} 
\item[$(a)$] Every net in $X$ $\mathbb{c_a}$-converges to every $x\in X$.
If $X\ne\{0\}$ then $\mathbb{c_a}$ is a linear non-$T_1$ convergence on $X$. 
\item[$(b)$] 
Eventually constant nets and only they are $\mathbb{c_d}$-convergent. The class 
$\mathbb{c_d}$ is a $T_1$ non-linear convergence on $X$.
\item[$(c)$] 
It is a straightforward fact that $X$ is both $\mathbb{c_a}$- and $\mathbb{c_d}$-complete.
\end{enumerate}
\end{example}

The present paper contains six further sections. Section 2 contains definitions of full lattice convergence, their elementary properties,
related examples, and Theorem \ref{fullification} on fullification of a linear convergence.
In Section 3, we present Theorem \ref{double nets}, that generalizes the Roberts -- Namioka theorem 
to the lattice convergence, and discuss several related results. Section 4 is devoted to the absolute $\mathbb{c}$-convergence and
the unbounded $\mathbb{c}$-convergence.
In Section 5, we introduce an $\mathbb{mc}$-convergence on a commutative $l$-algebra and study its general properties.
Section 6 is devoted to the $\mathbb{o}$-convergence on commutative $l$-algebras with especial emphasis at the case of 
universally complete $f$-algebras. In the final section, we consider the ${\text{\boldmath{$m\tau$}}}$-convergence on a commutative 
locally full $l$-algebra and the topology $\tau_m$ related to the convergence ${\text{\boldmath{$m\tau$}}}$.

We refer: for ordered spaces, Riesz spaces, and applications of Boolean-valued analysis to \cite{Vu,LZ,Go1,Go3,Za,AB,Ku,AT,GKK,AGG,AEEM2,AEEM1}; 
for $l$-algebras to \cite{Pag,Za,HP,Hu,BH,BR,BC,Ku,GKK}; for unbounded convergences to \cite{GTX,DEM2,DE,DEM3,Tay1,Tay2,AEEM2,AEEM1}.

\section{Full Lattice Convergence. Definitions and Examples.}

We begin with the following definition which is crucial for the present paper.

\begin{definition}\label{top+full+lattice convergence}
A convergence $\mathbb{c}$ on a vector space $X$ is called$:$
\begin{enumerate} 
\item[$(a)$] \ a {\em topological linear convergence}, if $\mathbb{c}$ is the convergence with respect to some linear topology on $X$$;$
\item[$(b)$] \ a {\em full convergence}, if $X$ is an ordered vector space, and $\mathbb{c}$ respects the order in $X$ in the following sense:
$$
  [(x_\alpha)_{\alpha\in A}\convc 0 \ \text{and} \ (\forall\alpha\in A) \ 0\le y_\alpha\le x_\alpha ]\Longrightarrow 
  (y_\alpha)_{\alpha\in A}\convc 0; 
  \eqno(1)
$$
\item[$(c)$] \ a  {\em lattice convergence}, if $X$ is a Riesz space, and $\mathbb{c}$ respects the lattice operations in $X$ in the following sense:
$$
  x_\alpha\convc x \Longrightarrow|x_\alpha|\convc|x|.
$$
\end{enumerate}
\end{definition}

\begin{remark}\label{linear topological convergence is T_1 iff topology is Hausdorff}
The following observations are trivial$:$
\begin{enumerate}
\item[$(i)$] \ if $(X,\tau)$ is a topological vector space, then the ${\text{\boldmath{$\tau$}}}$-convergence on $X$ is $T_1$ iff the topology $\tau$ is Hausdorff$;$
\item[$(ii)$] \ if $X$ is a Riesz space, then the convergence $\mathbb{c_a}$ on $X$ from Example $\ref{c_a and c_d}$ 
is  linear full lattice, whereas the convergence $\mathbb{c_d}$ is   full lattice which is non-linear when $X \ne\{0\}$.
\end{enumerate}
\end{remark}

\begin{proposition}\label{lattice in full}
Let $\mathbb{c}$ be an additive full convergence on a Riesz space $X$. The following conditions are equivalent$:$
\begin{enumerate}
\item[$(i)$] \ $\mathbb{c}$ is a lattice convergence$;$
\item[$(ii)$] for every net $x_\alpha$ in $X$,  $x_\alpha\convc 0\Longrightarrow|x_\alpha|\convc 0$.
\end{enumerate}
\end{proposition}

\begin{proof}
$(i)\Longrightarrow(ii)$ \ is trivially true for an arbitrary convergence $\mathbb{c}$ on $X$. 

$(ii)\Longrightarrow(i)$ \ Let $x_\alpha\convc x$ in $X$. Then $x_\alpha-x\convc 0$ since $\mathbb{c}$ is additive,
and hence $|x_\alpha-x|\convc 0$, by $(ii)$. Then $||x_\alpha|-|x||\le|x_\alpha-x|$ implies $||x_\alpha|-|x||\convc 0$
because $\mathbb{c}$ is full. Since $0\le(|x_\alpha|-|x|)^{\pm}\le||x_\alpha|-|x||\convc 0$ and $\mathbb{c}$ is full, we have $(|x_\alpha|-|x|)^{\pm}\convc 0$.
The additivity of $\mathbb{c}$ implies $|x_\alpha|-|x|=(|x_\alpha|-|x|)^{+}-(|x_\alpha|-|x|)^{-}\convc 0$ and then $|x_\alpha|\convc|x|$.
\end{proof}

We do not know, whether or not Proposition \ref{lattice in full} holds true if we drop the assumption that an additive convergence $\mathbb{c}$ is full.

\begin{proposition}\label{inequalities}
Let $\mathbb{c}$ be an additive lattice convergence on a Riesz space $X$. The following conditions are equivalent$:$
\begin{enumerate}
\item[$(i)$] \ $\mathbb{c}\in T_1$$;$
\item[$(ii)$] $\mathbb{c}$ preserves inequalities, i.e. if 
$(x_\alpha)_{\alpha\in A}\convc x,  (y_\alpha)_{\alpha\in A}\convc y$, and
$x_\alpha\ge y_\alpha$ for all $\alpha\in A$, then $x\ge y$$;$
\item[$(iii)$] \ the positive cone $X_+$ is $\mathbb{c}$-closed in $X$.
\end{enumerate}
Furthermore, $(ii)\Longleftrightarrow(iii)$ for an arbitrary additive convergence $\mathbb{c}$.
\end{proposition}

\begin{proof}
$(i)\Longrightarrow(ii)$ \ 
By Definition \ref{linear convergence}$(c)$,  $(y_\alpha-x_\alpha)_{\alpha\in A}\convc y-x$ and, by Definition \ref{top+full+lattice convergence}$(b)$,
$(|y_\alpha-x_\alpha|)_{\alpha\in A}\convc |y-x|$. Applying Definition \ref{linear convergence}$(c)$ once more, obtain
$$
   ((2\cdot 0)_{\alpha\in A}, \ 2(y-x)^+)=((2(y_\alpha-x_\alpha)^+)_{\alpha\in A}, \ 2(y-x)^+)=
$$
$$	
   ((|y_\alpha-x_\alpha|+y_\alpha-x_\alpha)_{\alpha\in A}, \ |y-x|+y-x))\in\mathbb{c}.
$$
Since $((2\cdot 0)_{\alpha\in A}, \ 0)\in\mathbb{c}$ by Definition \ref{convergence}$(a)$ 
and since $\mathbb{c}\in T_1$, it is true that $2(y-x)^+=0$, and hence $x\ge y$, as required.

$(ii)\Longrightarrow(i)$ \
Suppose on the contrary case, $\mathbb{c}\notin T_1$. There exist $a,b\in X$ such that $a\ne b$, yet $a\convc b$.
Hence $0=a-a\convc b-a$. Since $\mathbb{c}$ is a lattice convergence, $0\convc |a-b|$. 
Take $x_\alpha=y_\alpha\equiv 0$. Then $x_\alpha\convc x:=0$, $y_\alpha\convc y:=|a-b|\ne 0$.
Hence $x\not\ge y$, which violates $(ii)$. The obtained contradiction proves $\mathbb{c}\in T_1$. 

In the rest of the proof $\mathbb{c}$ is just an additive convergence on $X$. 

$(ii)\Longrightarrow(iii)$\ 
It follows from $0=:y_\alpha\le x_\alpha\convc x \Rightarrow 0\le x$. 

$(iii)\Longrightarrow(ii)$\  
Let $(x_\alpha)_{\alpha\in A}\convc x,  (y_\alpha)_{\alpha\in A}\convc y$, and $x_\alpha\ge y_\alpha$ for all $\alpha\in A$.
Then $(x_\alpha-y_\alpha)_{\alpha\in A}\convc x-y$ because $\mathbb{c}$ is additive. Since  $X_+$ is $\mathbb{c}$-closed in $X$ and 
$x_\alpha-y_\alpha\in X_+$ for all $\alpha\in A$, then $x-y\in X_+$ and hence $x\ge y$.
\end{proof}

\begin{example}\label{piece wise polynomial} 
Let $X=PP[0,1]$ be the Riesz space of continuous $\mathbb{R}$-valued functions on $[0,1]$
which are piecewise polynomial with finitely many pieces $($see {\em \cite[Ex.5.1(i)]{Hu}}$)$. We define 
$$
  \|f\|_{\mathbb{1}}:=\|f\|_{\infty}+\sup\bigg\{\bigg|\frac{df}{dt}(s)\bigg|: s\in [0,1] \ \text{and} \ \frac{df}{dt}(s)\ \text{exists} \ \bigg\}. \ \ \ (f\in X)
  \eqno(2)
$$ 
Let $\mathbb{c_1}$ be the convergence on $X$ with respect to the norm $\|\cdot\|_{\mathbb{1}}$ on $X$ given by $(2)$.
Clearly, $\mathbb{c_1}$ is a topological linear $T_1$ convergence on $X$. However $\mathbb{c_1}$ is neither lattice nor full on $X$. 
\begin{enumerate}
\item[$(a)$] \ 
The convergence $\mathbb{c_1}$ is not lattice on $X$. To see this, take $f_n = f-\frac{1}{n}$, $n\in \mathbb{N}$,  where
$$
    f(t) = \left\{ \begin{array}{lcl}
    0 & \text{if} & 0\leq t < \frac{1}{2} \\
    2t-1 & \text{if} & \frac{1}{2}\leq t \leq 1
    \end{array} \right. .
$$
Then $\| f_n - f\|_{\mathbb{1}} = \| f_n - f\|_{\infty} = \frac{1}{n} \to 0$,  but $\| |f_n| - |f|\|_{\mathbb{1}} = \frac{1}{n} + 4 \not\to 0$. 
\item[$(b)$] \ 
The convergence $\mathbb{c_1}$ is not full on $X$.  In order to show this, let
$y_n\in X$ be linear on every interval $[\frac{k-1}{2n},\frac{k}{2n}]$ for $k=1,\dotsc , 2n$, $y_n(\frac{k-1}{2n})=0$ for odd $k$,
and $y_n(\frac{k-1}{2n})=\frac{1}{n}$ for even $k$. Let $x_n(t)\equiv\frac{1}{n}$ on $[0,1]$. Then $0\le y_n\le x_n\convcone 0$,
however $\|y_n\|_{\mathbb{1}}= \frac{1}{n} + 2  \not\to 0$. 
\end{enumerate}
\noindent
It is less trivial to construct a linear convergence which is lattice but not full. In order to do this, take the vector subspace 
$Y$ of $X$ consisting of continuous on $[0,1]$ functions which are affine on $[\frac{1}{k+1},\frac{1}{k}]$ for each $k\in\mathbb{N}$.
Then $Y$ is a Riesz space with respect to the pointwise order. It is easy to see that the absolute value $|f|_Y$ of $f\in Y$ is given by
$$
    |f|_Y(t)=\bigg|f\bigg(\frac{1}{k+1}\bigg)\bigg|+\frac{(|f(\frac{1}{k})|-|f(\frac{1}{k+1})|)(t-\frac{1}{k+1})}{\frac{1}{k}-\frac{1}{k+1}} \ \ \ \ \ \ \bigg(t\in\bigg[\frac{1}{k+1},\frac{1}{k}\bigg]\bigg),
$$ 
where $k\in\mathbb{N}$, and $|f|_Y(0)=|f(0)|$. Clearly $|f|_Y\ne |f|$ in general, and hence the Riesz space $Y$ is not a Riesz subspace of $X$.
Furthermore, $\mathbb{c_1}$ is a topological linear $T_1$ convergence on $Y$.
\begin{enumerate}
\item[$(c)$] \  
The convergence $\mathbb{c_1}$ is a lattice convergence on the Riesz space $Y$. To see this, take $f_n, f\in Y$  such that $f_n\convcone f$ in  $Y$.
In other words:
$$
    \| f_n - f\|_{\mathbb{1}}=  \|f_n-f\|_{\infty}+
$$
$$
    \sup\bigg\{\bigg|\frac{d(f_n-f)}{dt}(s)\bigg|: s\in [0,1] \ \text{and} \ \frac{d(f_n-f)}{dt}(s)\ \text{exists} \ \bigg\}=
$$
$$
    \sup_{k\in\mathbb{N}}\bigg|f_n\bigg(\frac{1}{k}\bigg)-f\bigg(\frac{1}{k}\bigg)\bigg|+\sup_{k\in\mathbb{N}}\bigg\{\frac{|(f_n-f)(\frac{1}{k})-(f_n-f)(\frac{1}{k+1})|}{\frac{1}{k}-\frac{1}{k+1}} \bigg\}\to 0.
$$
So we need to show the same for modulus, namely $|f_n|\convcone |f|$ in  $Y$. 
$$
    \||f_n|_Y-|f|_Y \|_{\mathbb{1}}= \||f_n|_Y-|f|_Y \|_{\infty} + 
$$
$$
     \sup\bigg\{\bigg|\frac{d(|f_n|_Y-|f|_Y)}{dt}(s)\bigg|: s\in [0,1] \ \text{and} \ \frac{d(|f_n|_Y-|f|_Y)}{dt}(s)\ \text{exists} \ \bigg\}.
$$
By our assumption, all the functions are affine in each interval $\left[\frac{1}{k+1},\frac{1}{k}\right]$ for each $k\in\mathbb{N}$.
Therefore, we can restrict ourselves to just one interval, say $[a,b]$ with $0 < a < b \le 1$. 
In view of the definition of $| \cdot |_Y$, we can omit the condition of existence of $|f_n|_Y'$ and $|f|_Y'$. 
The following inequality $\| |f_n|'_Y - |f|'_Y\|_{\infty} \le 2  \| f'_n - f'\|_{\infty}$, whose proof is left to the reader, gives $|f_n|\convcone |f|$. 
\item[$(d)$] \  
The convergence $\mathbb{c_1}$ is not full on $Y$. To see this, let $f_n, g\in Y$ be such that
$$
    f_n(t) = \left\{ \begin{array}{lcl}
    \frac{1}{n(n+1)} & \text{if} & 0\le t \le\frac{1}{n+1} \\
    \text{affine} \      & \text{if} & \frac{1}{n+1}\le t \le \frac{1}{n} \\
    0                         & \text{if}  & \frac{1}{n}\leq t \leq 1
    \end{array} \right.,
$$
and $g_n(t)=\frac{1}{n}$ for all $t\in [0,1]$. Then $0\le f_n\le g_n\convcone 0$,
however $\|f_n\|_{\mathbb{1}}= \frac{1}{n(n+1)} + 1  \not\to 0$.
\end{enumerate}
\end{example}

Up to now, we have just one modification $\mathbb{c}\to\mathbb{sc}$ of convergence on a Riesz space $X$ or, in different terminology, a functor from the category of all  
convergences on $X$ into itself. The following definition presents another modification that transforms linear convergences to full lattice convergences.

\begin{definition}\label{lattice to full lattice}
Let $\mathbb{c}$ be a convergence on a Riesz space $X$. The {\em fullification} $\mathbb{fc}$ of $\mathbb{c}$ is defined by
$$
  ((y_\alpha)_{\alpha\in A}, y)\in\mathbb{fc} \ \text{if} \ \ (\exists ((x_\alpha)_{\alpha\in A}, 0)\in\mathbb{c})(\forall\alpha\in A)  |y_\alpha-y|\le x_\alpha.
$$
\end{definition}

The following theorem justifies the use of the term ``fullification" $\mathbb{fc}$ of a linear convergence $\mathbb{c}$ on a Riesz space.

\begin{theorem}\label{fullification}
Let $\mathbb{c}$ be a linear convergence on a Riesz space $X$. Then the class $\mathbb{fc}$ is a linear full lattice convergence on $X$.
Furthermore, if the convergence $\mathbb{c}$ is  lattice, then $\mathbb{c}\subseteq\mathbb{fc}$.
\end{theorem}

\begin{proof}
We omit the straightforward verification that $\mathbb{fc}$ satisfies conditions $(a)$, $(b)$, and $(c)$ of Definition \ref{convergence} and therefore is a convergence on $X$.

In order to show that $\mathbb{fc}$ is linear, let $(x_\alpha)_{\alpha\in A}\convfc x$, $(y_\beta)_{\beta\in B}\convfc y$ in $X$, and
$(r_\gamma)_{\gamma\in\Gamma}\to r\in\mathbb{R}$ in the standard topology on $\mathbb{R}$.
By Definition \ref{lattice to full lattice}, there exist nets $(z_\alpha)_{\alpha\in A}\convc 0$ and $(w_\beta)_{\beta\in B}\convc 0$ in $X$, such that 
$$
    |x_\alpha-x|\le z_\alpha \ \text{and} \ |y_\beta-y|\le w_\beta \ \ \ \ (\forall \alpha\in A,\beta\in B). 
$$
It follows that
$$
    |(x_\alpha+y_\beta)-(x+y)|\le|x_\alpha-x|+|y_\beta-y|\le z_\alpha+w_\beta \ \ \ \ (\forall \alpha\in A,\beta\in B). 
$$
Since the convergence $\mathbb{c}$ is linear, $(z_\alpha+w_\beta)_{(\alpha,\beta)\in A\times B}\convc 0$ and hence $(x_\alpha+y_\beta)_{(\alpha,\beta)\in A\times B}\convfc x+y$
by Definition \ref{lattice to full lattice}. So, the convergence $\mathbb{fc}$ satisfies Definition \ref{linear convergence}$(a)$.
For verifying Definition \ref{linear convergence}$(b)$, take $\alpha_0\in A$ and let $\gamma_0\in\Gamma$ be such that $|r_\gamma|\le|r|+1$ for all $\gamma\ge\gamma_0$.
Observe that
$$
   |r_\gamma x_\alpha-rx|\le|r_\gamma||x_\alpha-x|+|r_\gamma x -rx|\le
$$
$$
   (|r|+2)(|x_\alpha-x|)+|r_\gamma -r||x|\le(|r|+2)z_\alpha+|r_\gamma -r||x| 
  \eqno(3)
$$
for all $\alpha\ge\alpha_0$ in $A$ and for all $\gamma\ge\gamma_0$ in $\Gamma$. Since $\mathbb{c}$ is linear,
$$
  ((|r|+2)z_\alpha+|r_\gamma -r||x|)_{(\gamma,\alpha)\ge(\gamma_0,\alpha_0)\in \Gamma\times A}\convc 0.
  \eqno(4)
$$
By Definition \ref{lattice to full lattice}, $(3)$ and $(4)$ imply $(r_\gamma x_\alpha)_{(\gamma,\alpha)\ge(\gamma_0,\alpha_0)\in \Gamma\times A}\convfc rx$ and,
since $\mathbb{fc}$ is a convergence on $X$, there holds $(r_\gamma x_\alpha)_{(\gamma,\alpha)\in\Gamma\times A}\convfc rx$ by Definition \ref{convergence}$(c)$.
We have shown that $\mathbb{fc}$ is a linear convergence on $X$.

Now, let $(x_\alpha)_{\alpha\in A}\convfc 0$ and $0\le y_\alpha\le x_\alpha$ for all $\alpha\in A$.
By Definition \ref{lattice to full lattice}, there exists a net $(z_\alpha)_{\alpha\in A}\convc 0$ such that 
$|y_\alpha-0|=y_\alpha\le x_\alpha=|x_\alpha-0|\le z_\alpha $ for all $\alpha\in A$, that shows $(y_\alpha)_{\alpha\in A}\convfc 0$.
Thus, we have shown that $\mathbb{fc}$ is full.

Finally, let $(x_\alpha)_{\alpha\in A}\convfc x$. Take a net $(z_\alpha)_{\alpha\in A}\convc 0$ such that $|x_\alpha-x|\le z_\alpha$ for all $\alpha\in A$. 
In view of Definition \ref{lattice to full lattice}, $||x_\alpha|-|x||\le|x_\alpha-x|\le z_\alpha\convc 0$ implies $(|x_\alpha|)_{\alpha\in A}\convfc|x|$
and therefore $\mathbb{fc}$ is lattice.

It remains to prove that $\mathbb{c}\subseteq\mathbb{fc}$ assuming $\mathbb{c}$ to be a lattice convergence. Let $(y_\alpha)_{\alpha\in A}\convc y$. 
Then $(y_\alpha-y)_{\alpha\in A}\convc 0$ since $\mathbb{c}$ is linear. Hence $(|y_\alpha-y|)_{\alpha\in A}\convc 0$ since $\mathbb{c}$ is lattice. 
Letting $x_\alpha:=|y_\alpha-y|$ for all $\alpha\in A$ and applying Definition \ref{lattice to full lattice}, we obtain $(y_\alpha)_{\alpha\in A}\convfc y$.
\end{proof}

\begin{corollary}\label{fc=c for full lattice c}
Let $\mathbb{c}$ be a linear full lattice convergence on a Riesz space $X$. Then $\mathbb{fc}=\mathbb{c}$. 
\end{corollary}

\begin{proof}
Since the convergence $\mathbb{c}$ is lattice, Theorem \ref{fullification} implies $\mathbb{c}\subseteq\mathbb{fc}$. Let now $(x_\alpha)_{\alpha\in A}\convfc x$.
By Definition \ref{lattice to full lattice}, there exists a net $(z_\alpha)_{\alpha\in A}\convc 0$ such that 
$$
   0\le(x_\alpha-x)^{\pm}\le|x_\alpha-x|\le z_\alpha \ \ \ \ (\forall \alpha\in A).
$$
Since $\mathbb{c}$ is full, $(x_\alpha-x)^{\pm}\convc 0$ and hence $x_\alpha-x=(x_\alpha-x)^{+}-(x_\alpha-x)^{-}\convc 0$.
Since $\mathbb{c}$ is linear, $(x_\alpha)_{\alpha\in A}\convc x$. Hence $\mathbb{fc}\subseteq\mathbb{c}$, which completes the proof.
\end{proof}

\begin{corollary}\label{ffc=fc}
Let $\mathbb{c}$ be a linear convergence on a Riesz space $X$. Then $\mathbb{ffc}=\mathbb{fc}$.
\end{corollary}

\begin{proof}
By Theorem \ref{fullification}, $\mathbb{fc}$ is a linear full lattice convergence. Application of Corollary~\ref{fc=c for full lattice c}  to $\mathbb{fc}$ completes the proof.
\end{proof}

In Sections 4, 5, and 6 we introduce and investigate several important modifications of a linear full convergence on a Riesz space $X$ 
which produce new linear full lattice convergences.

The following useful fact is well known (see, e.g., \cite[Thm.2.23]{AT}). For convenience, we include its proof in our terminology.

\begin{lemma}\label{full topology}
Let $\tau$ be a linear topology on an ordered vector space $X$. Then the following conditions are equivalent$:$
\begin{enumerate}
\item[$(i)$] \ the topology $\tau$ is locally full$;$
\item[$(ii)$] \ the ${\text{\boldmath{$\tau$}}}$-convergence is linear and full.
\end{enumerate}
\end{lemma}

\begin{proof}
$(i)\Longrightarrow(ii)$ \ is trivial.

$(ii)\Longrightarrow(i)$ \ 
Denote by $\tau(0)$ the collection of all  $\tau$-neighborhoods of zero. By way of contradiction, assume that $(ii)$ is true, 
and $\tau$ is not locally full. Then there exists $U\in\tau(0)$ such that the {\em full hull} $[V]:=\bigcup\limits_{x,y\in V}[x,y]$ 
of every $V\in\tau(0)$ is not in $U$. So, for each $V\in\tau(0)$, we can pick vectors $x_V,y_V\in V$ and $z_V\in X\setminus V$ satisfying $x_V\le z_V\le y_V$.
Consider $\tau(0)$ as a directed set under the ordering $V_1\le V_2$ whenever $V_2\subseteq V_1$. Clearly
$$
   (x_V)_{V\in\tau(0)}\convtau 0 \ , \  (y_V)_{V\in\tau(0)}\convtau 0 \ , \  \ \text{and} \ \ \ 
   (y_V-x_V)_{V\in\tau(0)}\convtau 0 .
   \eqno(5)
$$
Due to $(5)$ and the inequality
$$
    0\le z_V - x_V \le y_V - x_V \ \ \ \ (\forall\ V\in\tau(0)),
$$
the condition $(ii)$ implies $z_V-x_V\convtau 0$ and hence $z_V=(z_V-x_V)+x_V\convtau 0$,  contrary to $z_V\notin U$  for all $V\in\tau(0)$.
The obtained contradiction completes the proof.
\end{proof} 

\begin{example}\label{full but not lattice} 
Consider the weak convergence $\mathbb{w}$ on the Banach lattice $X=L_\infty[0,1]$. Clearly $\mathbb{w}$
is a linear convergence. Furthermore, $\mathbb{w}$ is full. Indeed, if $0\le y_\alpha\le x_\alpha\convw 0$ then 
$$
  0\le|f(y_\alpha)|\le|f|(y_\alpha)\le|f|(x_\alpha)\to 0 \ \ \ \ (\forall f\in X'),
$$
and hence $y_\alpha\convw 0$. However, the convergence $\mathbb{w}$ is not a lattice convergence because, 
for a sequence $x_n\in X$ of Rademacher's functions on $[0,1]$, $x_n\convw 0$ yet $|x_n|\convw 1$. 
\end{example}

Recall that a net $x_\alpha$ in a Riesz space $X$ is {\em order convergent} to $x\in X$ if there is a net $y_\beta\downarrow 0$ 
in $X$ such that, for any $\beta$, there is $\alpha_{\beta}$ with $|x_\alpha-x|\le y_\beta$ for all $\alpha\ge \alpha_{\beta}$.

Remind that any Archimedean Riesz space $X$ has unique up to Riesz homomorphism Dedekind completion which will be denoted by $X^\delta$.

The {\em order convergence} (briefly, {\em $\mathbb{o}$-convergence}) is one of the most important modes of convergence on a Riesz space. 
The $\mathbb{o}$-convergence serves (together with the ${\text{\boldmath{$\tau$}}}$-convergence on a locally solid Riesz space $(X,\tau)$) 
the major motivation for Definition \ref{top+full+lattice convergence}. 

\begin{proposition}\label{o-convergence is full lattice and T_1}
The $\mathbb{o}$-convergence on a Riesz space $X$ is an additive full lattice $T_1$ convergence.
Furthermore,  $\mathbb{o}$-convergence on $X$ is linear iff $X$ is Archimedean. 
\end{proposition}

\begin{proof}
The additivity of $\mathbb{o}$ follows from \cite[Thm.1.14(ii)]{AB}. The fact that $\mathbb{o}$ is full is an immediate consequence of the definition of order convergence.
The convergence $\mathbb{o}$ is lattice by \cite[Thm.1.14(i)]{AB} and $\mathbb{o}$ is $T_1$ by \cite[Lm.1.13]{AB}.

Let $X$ be Archimedean. Since $\mathbb{o}$-convergence is additive full lattice, for proving that $\mathbb{o}$ is linear it is enough to check $(a)$ and $(b)$ of Definition \ref{linear convergence}.
Let $(x_\alpha)_{\alpha\in A}\convo x$ and $(y_\beta)_{\beta\in B}\convo y$ in $X$, and $(r_\gamma)_{\gamma\in\Gamma}\to r\in\mathbb{R}$ in the standard topology on $\mathbb{R}$. 
Let $\gamma_0\in\Gamma$ be such that $|r_\gamma|\le|r|+1$ for all $\gamma\ge\gamma_0$. For every $\alpha\in A$ and $\beta\in B$,
$|x_\alpha+y_\beta-(x+y)|\le|x_\alpha-x|+|y_\beta-y|$.  Definition \ref{linear convergence}$(c)$ together with conditions $(c)$ and $(b)$ of Definition \ref{top+full+lattice convergence}
imply $(|x_\alpha+y_\beta-(x+y)|)_{(\alpha,\beta)\in A\times B}\convo 0$. By the above inequality, it follows that $(x_\alpha+y_\beta)_{(\alpha,\beta)\in A\times B}\convo x+y$. Thus, $\mathbb{o}$ 
satisfies $(a)$ of Definition \ref{linear convergence}. Similarly, for every $\gamma\ge\gamma_0$ and $\alpha\in A$,
$$
  |r_\gamma x_\alpha-rx|\le|r_\gamma||x_\alpha-x|+|r_\gamma x -rx|\le(|r|+2)(|x_\alpha-x|)+|r_\gamma -r||x|.
$$
Since $X$ is Archimedean, $(|r_\gamma -r||x|)_{\gamma\in\Gamma}\convo 0$. By Definition \ref{linear convergence}$(b)$ and Definition \ref{top+full+lattice convergence}$(b)$, 
the above inequality implies $(|r_\gamma x_\alpha - rx|)_{\gamma\in\Gamma;\gamma\ge\gamma_0;\alpha\in A}\convo 0$, and
hence $(r_\gamma x_\alpha)_{\gamma\in\Gamma;\alpha\in A}\convo rx$. Therefore, $\mathbb{o}$ satisfies $(b)$ of Definition \ref{linear convergence} and hence $\mathbb{o}$ is linear.

Let $\mathbb{o}$-convergence on $X$ be linear. Assume on contrary that $X$ is non-Archimedean. There exist $x,y\in X$ such that $0 < x \le \frac{1}{n}\cdot y$ for 
all $n \in \mathbb{N}$. Since $y \convo y$ in $X$ and $\frac{1}{n} \to 0$ in  $\mathbb{R}$, the assumption that $\mathbb{o}$ is linear implies $\frac{1}{n}\cdot y\convo 0\cdot y = 0$
contradicting to the fact that $0 < x \le \frac{1}{n}\cdot y=\frac{1}{n}\cdot y$ for all $n \in \mathbb{N}$. The obtained contradiction shows that $X$ is Archimedean.
\end{proof}   

Recall that a net $x_{\alpha}$ in a Riesz space $X$ {\em relatively uniform converges} to $x\in X$
($x_{\alpha}\convr x$, for short) if there exists $u\in X_+$, such that, for any $n\in\mathbb{N}$, 
there exists $\alpha_n$ such that $|x_\alpha-x|\le\frac{1}{n} u$ for all $\alpha\ge\alpha_n$ (see, for example, \cite[1.3.4, p.20]{Ku}). 
Similarly to Proposition \ref{o-convergence is full lattice and T_1}, it follows directly from definition of relatively uniform convergence above
that  $\mathbb{r}$-convergence is additive, full, and lattice.
The following fact is well known in a different terminology. We include its elementary proof in our setting.

\begin{lemma}\label{r convergence is T_1 iff X is Archimedean}
Let $X$ be a Riesz space. The following conditions are equivalent$:$
\begin{enumerate}
\item[$(i)$] \ the Riesz space $X$ is Archimedean$;$
\item[$(ii)$] \ the convergence $\mathbb{r}$ on $X$ is $T_1$$;$
\item[$(iii)$] \ $\mathbb{r}\subseteq\mathbb{o}$ holds true on $X$.
\end{enumerate}
\end{lemma}

\begin{proof}
$(i)\Longrightarrow(ii)$ \  
Let $X$ be Archimedean. Take a constant net $a_\alpha\equiv a\in X$ and $b\ne a$ in $X$. In view of Lemma \ref{T_1 for linear conv}, 
for proving that $\mathbb{r}\in T_1$, it is enough to show that $a_\alpha$ does not $\mathbb{r}$-converge to $b$. Assume, on the contrary,
$a_\alpha\xrightarrow{\mathbb{r}}b$. By Definition \ref{convergence}$($a$)$, $(a_\alpha)_\alpha\convr a$. At the same time, 
in view of Definition \ref{linear convergence}$($c$)$, $(0)_\alpha=(a_\alpha-a_\alpha)_\alpha\convr a-b$, which provides  existence of $u\in E_+$ with
$|a-b|=|a-b-0|\le\frac{1}{n}u$ for any $n\in\mathbb{N}$. Since $X$ is Archimedean, $|a-b|=0$, violating $b\ne a$. The obtained contradiction shows that 
$a_\alpha$ does not $\mathbb{r}$-converge to $b$, and hence $\mathbb{r}\in T_1$, by Lemma \ref{T_1 for linear conv}.

$(ii)\Longrightarrow(i)$ \  
Let $\mathbb{r}\in T_1$ on $X$ and $x, y\in X$ satisfy $0\le x<\frac{1}{n} y$ for all $n\in\mathbb{N}$. Thus, the constant 
sequence $(x)_{n\in\mathbb{N}}\convr 0$, since $|x-0|=x<\frac{1}{n} y$ for all $n\in\mathbb{N}$. By $($a$)$ and $($d$)$ of Definition \ref{convergence}, 
we get $x=0$ because of $(x)_{n\in\mathbb{N}}\convr x$. This shows that $X$ is Archimedean.

$(i)\Longrightarrow(iii)$ \  Let $X$ be Archimedean and $((x_\alpha)_\alpha, x)\in\mathbb{r}$. Take $u\in X_+$ such that, 
for any $n\in\mathbb{N}$, there exists $\alpha_n$ such that $|x_\alpha-x|\le\frac{1}{n} u$ for all $\alpha\ge\alpha_n$.
Since $X$ is Archimedean, $\frac{1}{n}u\downarrow 0$ and hence $((x_\alpha)_\alpha, x)\in\mathbb{o}$.

$(iii)\Longrightarrow(ii)$ \ By Proposition \ref{o-convergence is full lattice and T_1}, $\mathbb{o}\in T_1$. In view of $\mathbb{r}\subseteq\mathbb{o}$,
Lemma \ref{T_1 for linear conv}  ensures that  $\mathbb{r}\in T_1$, as required. 
\end{proof}

\begin{proposition}\label{r convergence is sequential}
The relatively uniform convergence is sequential on any Riesz space $X$. 
\end{proposition}

\begin{proof}
Accordingly to Definition \ref{sequential c-convergence}$(b)$, we have to show that $\mathbb{r}=\mathbb{sr}$.

First we show $\mathbb{r}\subseteq\mathbb{sr}$. Let $(x_{\alpha})_{\alpha\in A}\convr x\in X$ and
$(x_{\alpha_\beta})_{\beta\in B}$ be a subnet of $(x_{\alpha})_{\alpha\in A}$. Then $(x_{\alpha_\beta})_{\beta\in B}\convr x$
by Definition \ref{convergence}$(b)$. By the definition of relatively uniform convergence, there exist $u\in X_+$ 
and a sequence $\beta_n\in B$ such that $|x_{\alpha_\beta}-x|\le\frac{1}{n}u$ for all $\beta\ge\beta_n$. 
Then $(x_{\alpha_{\beta_n}})_{n\in\mathbb{N}}\convr x$. Since the subnet $(x_{\alpha_\beta})_{\beta\in B}$ of $(x_{\alpha})_{\alpha\in A}$ is arbitrary, 
Definition \ref{sequential c-convergence}$(a)$ implies $(x_{\alpha})_{\alpha\in A}\convsr x$.
 
Now we show $\mathbb{sr}\subseteq\mathbb{r}$. Assume on the contrary that there exists a net $(x_{\alpha})_{\alpha\in A}$ such that
 $(x_{\alpha})_{\alpha\in A}\convsr x\in X$ but
$x_{\alpha}\stackrel{\mathbb{r}}{\nrightarrow} x$. Then, for every $u\in X_+$ and every $n\in\mathbb{N}$, there exists a subnet
$(x_{\alpha_\beta})_{\beta\in B}$ of $(x_{\alpha})_{\alpha\in A}$ such that 
$|x_{\alpha_\beta}-x|\not\le\frac{1}{n}u$ for all $\beta\in B$. Since $\mathbb{sr}\subseteq\mathbb{r}$, 
there exist $u\in X_+$ and a sequence $\beta_n\in B$ satisfying $|x_{\alpha_{\beta_n}}-x|\le\frac{1}{n}u$ for all natural $n$, 
a contradiction. This shows $\mathbb{sr}\subseteq\mathbb{r}$.
\end{proof}

The following proposition is well known (cf., e.g., \cite[Prop.2.4]{W1}). 
It shows that, in Riesz spaces, the notion of $\mathbb{r}$-completeness coincides with the notion of {\em ``sequential $\mathbb{r}$-completeness''}.
For the sake of convenience we include its short elementary proof. 
 
\begin{proposition}\label{r-completeness is sequential}
Let $X$ be a Riesz space. The following conditions are equivalent$:$
\begin{enumerate}
\item[$(i)$] \  $X$ is $\mathbb{r}$-complete$;$
\item[$(ii)$] \ for every $\mathbb{r}$-Cauchy sequence $x_n$ in $X$, there exists $x\in X$ with $x_n\convr x$. 
\end{enumerate}
\end{proposition}

\begin{proof}
$(i)\Longrightarrow(ii)$ is trivial.\\
$(ii)\Longrightarrow(i)$ Let $(x_\alpha)_{\alpha\in A}$ be an $\mathbb{r}$-Cauchy net in $X$. Then, for some $u\in X_+$ and for each  $n\in\mathbb{N}$, 
there exists $\alpha_n\in A$ such that $|x_{\alpha}-x_{\alpha'}|\le\frac{1}{n}u$ for all $\alpha,\alpha'\ge\alpha_n$. Without loss of generality, we may 
assume $\alpha_n\le\alpha_{n+1}$ for all  $n\in\mathbb{N}$. Clearly $x_{\alpha_n}$ is an $\mathbb{r}$-Cauchy sequence in $X$. By $(ii)$, $x_{\alpha_n}\convr x$ for some $x\in X$.
Then, for some $v\in X_+$, there exists $x\in X$ such that $|x_{\alpha_k}-x|\le\frac{1}{n}v$ for all $k\ge n$. The following 
$$
  |x_\alpha-x|\le|x_\alpha-x_{\alpha_n}|+|x_{\alpha_n}-x|\le\frac{1}{n}(u+v)  \ \ \ \ (\forall\alpha\ge\alpha_n)
$$
implies $x_\alpha\convr x$, as desired.
\end{proof}

\begin{remark}\label{r-completion}
Every Archimedean Riesz space $X$ has a unique $($up to a Riesz isomorphism$)$ 
{\em $\mathbb{r}$-completion} $X^r$. It can be identified with
$$
  \bigcap\{Y|Y \text{is} \ \mathbb{r}\text{-complete and X is a Riesz subspace of} \ Y\subseteq X^\delta\},
$$ 
which is well defined, because $X^\delta$ is $\mathbb{o}$-complete and hence $\mathbb{r}$-complete.
\end{remark}

As an example of a linear full lattice $T_1$ convergence on a Riesz space, we also mention the $p$-convergence from \cite{AEEM1,AEEM2}.

\section{General Properties of Full Convergences on Riesz Spaces}

The $\mathbb{o}$-convergence on a Riesz space $X$ is topological only if $\dim(X)<\infty$ \cite[Thm.1]{DEM3} (cf. also \cite{Gor}).
This fact emphasizes the importance of investigating convergences on Riesz spaces by using of general rather than only topological methods.
In Theorem \ref{double nets} below we give a characterization of the lattice convergence in the class of linear convergences on a Riesz space.

Firstly, we recall the following inequality
$$
  |u\vee v-f\vee g|\vee|u\wedge v-f\wedge g|\le|u-f|+|v-g|\ \ \ \ (\forall u,v,f,g\in X),
  \eqno(7)
$$
which is a consequence of the Bir\-khoff identity (see, e.g.,  \cite[Thm.12.4]{LZ}):
$$
  |f\vee h-g\vee h|+|f\wedge h-g\wedge h|=|f-g|\ \ \ \ (\forall f,h,g\in X). 
$$
Indeed, 
$$
  |u\vee v-f\vee g|\le|u\vee v-v\vee f|+|f\vee v-f\vee g|\le|u-f|+|v-g|
$$
and similarly $|u\wedge v-f\wedge g|\le|u-f|+|v-g|$ which constitutes (7).

\begin{theorem}\label{double nets}
For a linear convergence $\mathbb{c}$ on a Riesz space $X$, the following five conditions are equivalent$:$
\begin{enumerate}
\item[$(i)$] \ $\mathbb{c}$ is a lattice convergence$;$
\item[$(ii)$] \ if $x_\alpha\convc x$ then $x_\alpha^+\convc x^+$$;$
\item[$(iii)$] \ if $x_\alpha\convc x$ then $x_\alpha^-\convc x^-$$;$
\item[$(iv)$] \ if $x_\alpha\convc x$ and $y\in X$ then $x_\alpha\vee y\convc x\vee y$$;$
\item[$(v)$] \ if $x_\alpha\convc x$ and $y\in X$ then $x_\alpha\wedge y\convc x\wedge y$.
\end{enumerate}
\noindent
Furthermore, if $\mathbb{c}$ is also full then the conditions $(i)-(v)$ are equivalent to each of the following two conditions$:$
\begin{enumerate}
\item[$(vi)$] \ if $(x_\alpha)_{\alpha\in A}\convc x$ and $(y_\beta)_{\beta\in B}\convc y$ then
$(x_\alpha\wedge y_\beta)_{(\alpha,\beta)\in A\times B}\convc  x\wedge y$$;$
\item[$(vii)$] \ if $(x_\alpha)_{\alpha\in A}\convc x$ and $(y_\beta)_{\beta\in B}\convc y$ then
$(x_\alpha\vee y_\beta)_{(\alpha,\beta)\in A\times B}\convc x\vee y$.
\end{enumerate}
\end{theorem}

\begin{proof}
While proving, we use freely Definition \ref{linear convergence}. 

$(i)\Longrightarrow(ii)$ \ Use the identity $f^+=\frac{1}{2}|f|+\frac{1}{2}f$.

$(ii)\Longrightarrow(iii)$ \ Use the identity $f^-=(-f)^+$.

$(iii)\Longrightarrow(iv)$ \ Use the identity $f\wedge g=g-(f-g)^-$.

$(iv)\Longrightarrow(v)$ \ Use the identity $f\vee g=-((-f)\wedge(-g))$.

$(v)\Longrightarrow(i)$ \ Use the identity $|x|=x\vee (-x)$.

\noindent
From now on in the proof we suppose $\mathbb{c}$ to be linear and full.

$(i)\Longrightarrow(vi)$ \  
Let $(x_\alpha)_{\alpha\in A}\convc x$ and $(y_\beta)_{\beta\in B}\convc y$.
By Definition \ref{top+full+lattice convergence}$(c)$,
$$
  (|x_\alpha-x|+|y_\beta-y|)_{(\alpha,\beta)\in A\times B}\convc|x-x|+|y-y|=0.
$$
In view of $(7)$, we have
$$
  |x_\alpha\wedge y_\beta -x\wedge y|\le|x_\alpha-x|+|y_\beta-y|, 
  \eqno(8)
$$
and since $\mathbb{c}$ is full, $(8)$ implies 
$$
   (x_\alpha\wedge y_\beta)_{(\alpha,\beta)\in A\times B}\convc x\wedge y
$$
which is required.

$(vi)\Longrightarrow(vii)$ \ Use again the identity $f\vee g=-((-f)\wedge(-g))$.

$(vii)\Longrightarrow(iv)$ \ Apply $(vii)$ to the constant net $y_\beta\equiv y$.
\end{proof}

Since the topological convergence with respect to a locally full topology on a Riesz space is linear and  full 
(cf. Lemma \ref{full topology}), the next classical result of Roberts \cite{Ro} and Namioka \cite[Thm.8.1]{Na} 
can be considered as a particular case of Theorem \ref{double nets}.

\begin{theorem}$($The Roberts -- Namioka  theorem$)$\label{Roberts -- Namioka  Theorem}
For a linear topology $\tau$ on a Riesz space $X$, the following five conditions are equivalent$:$
\begin{enumerate}
\item[$(i)$] \ if $x_\alpha\convtau x$ then $|x_\alpha|\convtau|x|$$;$
\item[$(ii)$] \ if $x_\alpha\convtau x$ then $(x_\alpha)^+\convtau x^+$$;$
\item[$(iii)$] \ if $x_\alpha\convtau x$ then $(x_\alpha)^-\convtau x^-$$;$
\item[$(iv)$] \ if $x_\alpha\convtau x$ and $y\in X$ then $x_\alpha\vee y\convtau x\vee y$$;$
\item[$(v)$] \ if $x_\alpha\convtau x$ and $y\in X$ then $x_\alpha\wedge y\convtau x\wedge y$.
\end{enumerate}
Furthermore, if the linear topology $\tau$ is locally full, then the conditions $(i)$~-- $(v)$ are equivalent to each of the following three conditions$:$
\begin{enumerate}
\item[$(vi)$] \ $(x_\alpha)_{\alpha\in A}\convtau x$ and $(y_\beta)_{\beta\in B}\convtau y$ imply
$(x_\alpha\wedge y_\beta)_{(\alpha,\beta)\in A\times B}\convtau x\wedge y$$;$
\item[$(vii)$] \ $(x_\alpha)_{\alpha\in A}\convtau x$ and $(y_\beta)_{\beta\in B}\convtau y$ imply
$(x_\alpha\vee y_\beta)_{(\alpha,\beta)\in A\times B}\convtau x\vee y$$;$
\item[$(viii)$] \  the topology $\tau$ is locally solid$;$
\item[$(ix)$] \  the lattice operations in $X$ are uniformly $\tau$-continuous.
\end{enumerate}
\end{theorem}

\begin{proof}
Since the ${\text{\boldmath{$\tau$}}}$-convergence on $X$ is linear, Theorem \ref{double nets} implies that 
$(i)\Longleftrightarrow(ii)\Longleftrightarrow(iii)\Longleftrightarrow(iv)\Longleftrightarrow(v)$.

From now on, in the proof, we suppose the topology $\tau$ to be locally full.
Then the ${\text{\boldmath{$\tau$}}}$-convergence is full (cf. Lemma \ref{full topology}). In view of Theorem \ref{double nets}, 
$(i)\Longleftrightarrow(ii)\Longleftrightarrow(iii)\Longleftrightarrow(iv)\Longleftrightarrow(v)\Longleftrightarrow(vi)\Longleftrightarrow(vii)$.  

$(vii)\Longrightarrow(viii)$ \ Here we repeat Namioka's arguments from his proof of the implication $(ii)\Longleftrightarrow(i)$ in \cite[Thm.8.1]{Na}.
Pick any $U\in\tau(0)$. Without loss of generality, $U$ is full and balanced. Let $V\in\tau(0)$ be such that $V+V\subset U$. Since the mappings $x\to x^{\pm}$
are ${\text{\boldmath{$\tau$}}}$-continuous by $(ii)$ and $(iii)$, there exists a balanced $W\in\tau(0)$ such that $x\in W$ implies 
$x^\pm\in V$, and hence $|x|=x^++x^-\in V+V\subset U$. If $y$ is in the solid hull $S(x)$ of $x$ then $-|x|\le y\le|x|$.
Since $U$ is full and balanced, $y\in U$. Clearly $\bigcup\limits_{x\in W}S(x)$ is a solid neighborhood of 0 contained in $U$.

$(viii)\Longrightarrow(i)$ \  It is a straightforward fact, that the topological convergence with respect to any locally solid topology satisfies $(i)$.

$(viii)\Longleftrightarrow(ix)$ \ This is exactly $\cite[Thm.2.17]{AB}$.
\end{proof}

\begin{remark}
\begin{enumerate} Let $\tau$ be a linear topology on a Riesz space $X$. 
\item[$(i)$] \ If $X$ is normed, arguing as in Example $\ref{full but not lattice}$, we obtain that the weak convergence $\mathbb{w}$ is  full. 
It is worth to mention that, by the Roberts -- Namioka theorem and the Peressini theorem $($cf. $\cite[Thm.2.38]{AB}$$)$, $\mathbb{w}$ is not a lattice convergence unless $\dim(X)<\infty$.  
\item[$(ii)$] \ By Example $\ref{piece wise polynomial}$$(c)\& (d)$, the local fullness assumption on $\tau$ is crucial for 
the equivalence of $(i)-(v)$ with $(vi)-(ix)$ in Theorem $\ref{Roberts -- Namioka  Theorem}$.
\item[$(iii)$] \ The equivalence $(v)\Longleftrightarrow(vi)$ in Theorem $\ref{Roberts -- Namioka  Theorem}$ can be expressed as follows:
{\em if the linear topology $\tau$ if locally full, then the infimum operation $\wedge: X\times X\to X$ is jointly ${\text{\boldmath{$\tau$}}}$-continuous iff it is separately 
${\text{\boldmath{$\tau$}}}$-continuous}. 
\end{enumerate}
\end{remark}

\begin{theorem}\label{locally solid convergence}
Let $\tau$ be a linear topology on a Riesz space $X$. Then  the following conditions are equivalent$:$
\begin{enumerate}
\item[$(i)$] \ the ${\text{\boldmath{$\tau$}}}$-convergence is a linear full lattice convergence$;$
\item[$(ii)$] \ the topology $\tau$ is locally solid.
\end{enumerate}
In particular, on a locally full Riesz space $(X,\tau)$ the ${\text{\boldmath{$\tau$}}}$-convergence is lattice iff the topology $\tau$ is locally solid.
\end{theorem}

\begin{proof}
$(i)\Longrightarrow(ii)$ \  By Lemma \ref{full topology}, $\tau$ is locally full. Since the ${\text{\boldmath{$\tau$}}}$-convergence is lattice, 
Theorem \ref{Roberts -- Namioka  Theorem} implies that the topology $\tau$ is locally solid.

$(ii)\Longrightarrow(i)$ \ If $\tau$ is a locally solid topology then $\tau$ is locally full, e.g., by \cite[Exer.1(c), p.72]{AB}. 
Lemma \ref{full topology} implies that ${\text{\boldmath{$\tau$}}}$-convergence is linear and full. It follows then from Theorem \ref{Roberts -- Namioka  Theorem}$(i)$, that the 
${\text{\boldmath{$\tau$}}}$-convergence is also lattice.
\end{proof}

\section{$\mathbb{ac}$- and $\mathbb{uc}$-convergences on Riesz Spaces}

Here we study two important operations on full convergences on Riesz spaces which transform them to full lattice convergences.
We begin with a generalization of the convergence with respect to the almost weak topology on Riesz spaces. 
We recall the following definition (see for example \cite[Def.2.32]{AB}).

\begin{definition}\label{aw-convergence}
Let ${\mathcal F}$ be a non-empty family of Riesz seminorms on a Riesz space $X$. The {\em absolute weak topology} 
$|\sigma|(X, {\mathcal  F})$ on $X$ is the topology generated by the family  of seminorms
$$
    \rho_{\phi}(x):=\phi(|x|) \ \ \ \ (\phi\in{\mathcal  F}, x\in X).
$$
\end{definition}

Observe that, in Definition \ref{aw-convergence}, we have two full topological convergences, $\mathbb{c}$ and $\mathbb{c}'$ on $X$ 
with respect to the topologies $\sigma(X, {\mathcal  F})$ and $\sigma(X, {\mathcal  G})$, where ${\mathcal  G}=\{\rho_{\phi}:\phi\in{\mathcal  F}\}$.
It is well known that the topology $\sigma(X, {\mathcal  G})$ is locally solid. This observation motivates the following definition.

\begin{definition}\label{ac-convergence}
Let $\mathbb{c}$ be a convergence on a Riesz space $X$. The absolute $\mathbb{c}$-convergence on $X$ $($briefly $\mathbb{ac}$$)$ is defined by
$$
    x_\alpha\convac x\in X \ \  \text{if} \ \ |x_\alpha-x|\convc 0.
$$
\end{definition}

It should be clear that $\mathbb{c}'=\mathbb{ac}$ in the notations of the observation before Definition \ref{ac-convergence}. More generally,

\begin{theorem}\label{ac convergence is full lattice}
Let $\mathbb{c}$ be a linear  full convergence on a Riesz space $X$. Then $\mathbb{ac}$ is a linear full lattice convergence on $X$. Furthermore, 
\begin{enumerate} 
\item[$(i)$] \ $\mathbb{ac}\subseteq\mathbb{c}$$;$
\item[$(ii)$] \ $\mathbb{aac}=\mathbb{ac}$.
\item[$(iii)$] \ $\mathbb{c}\in T_1 \ \Longrightarrow \ \mathbb{ac}\in T_1$.
\end{enumerate}
\end{theorem}

\begin{proof}
Clearly $\mathbb{ac}$ is a convergence on $X$. In order to show that $\mathbb{ac}$ is linear, let $(x_\alpha)_{\alpha\in A}\convac x$, $(y_\beta)_{\beta\in B}\convac y$, 
and $(r_\gamma)_{\gamma\in\Gamma}\to r\in\mathbb{R}$ in the standard topology on $\mathbb{R}$. Let $\gamma_0\in\Gamma$ be such that 
$|r_\gamma|\le|r|+1$ for all $\gamma\ge\gamma_0$. For every $\alpha\in A$ and $\beta\in B$,
$|x_\alpha+y_\beta-(x+y)|\le|x_\alpha-x|+|y_\beta-y|$. In view of Definition \ref{linear convergence}$(a)$ and Definition \ref{top+full+lattice convergence}$(b)$, 
it follows $(|x_\alpha+y_\beta-(x+y)|)_{(\alpha,\beta)\in A\times B}\convc0$, which means $(x_\alpha+y_\beta)_{(\alpha,\beta)\in A\times B}\convac x+y$.
Similarly, for every $\gamma\ge\gamma_0$ and $\alpha\in A$,
$$
  |r_\gamma x_\alpha-rx|\le|r_\gamma||x_\alpha-x|+|r_\gamma x -rx|\le(|r|+2)(|x_\alpha-x|)+|r_\gamma -r||x|.
$$
This inequality implies $(|r_\gamma x_\alpha - rx|)_{\gamma\in\Gamma;\gamma\ge\gamma_0;\alpha\in A}\convc 0$
by Definition \ref{linear convergence}$(b)$ and Definition \ref{top+full+lattice convergence}$(b)$. By Definition \ref{convergence}$(c)$, 
$(|r_\gamma x_\alpha - rx|)_{\gamma\in\Gamma;\alpha\in A}\convc 0$, which shows $(r_\gamma x_\alpha)_{\gamma\in\Gamma;\alpha\in A}\convac rx$. 
Therefore, $\mathbb{as}$ satisfies Definition \ref{linear convergence} and hence is a linear convergence on $X$.

Suppose $0\le y_\alpha\le x_\alpha\convac 0$. Then $0\le y_\alpha\le x_\alpha=|x_\alpha|\convc 0$. Since $\mathbb{c}$ is full then 
$|y_\alpha|=y_\alpha\convc 0$, which means $y_\alpha\convac 0$ and hence the convergence $\mathbb{ac}$ is full.
By Proposition \ref{lattice in full}, $\mathbb{ac}$ is a lattice convergence since $x_\alpha\convac 0$ trivially implies $|x_\alpha|\convac 0$.

$(i)$ Let $x_\alpha\convac x\in X$. Then $|x_\alpha-x|\convc 0$. Since $\mathbb{c}$ is full and $0\le(x_\alpha-x)^{\pm}\le|x_\alpha-x|$
then $(x_\alpha-x)=(x_\alpha-x)^++(x_\alpha-x)^-\convc 0$, and hence $x_\alpha\convc x$. We have shown $\mathbb{ac}\subseteq\mathbb{c}$.

$(ii)$ The formula $\mathbb{aac}=\mathbb{ac}$ follows by applying Definition \ref{ac-convergence} to $\mathbb{ac}$ which is already proved to be a full convergence.

$(ii)$ Let $\mathbb{c}\in T_1$. Then the inclusion $\mathbb{ac}\subseteq\mathbb{c}$ implies $\mathbb{ac}\in T_1$.
\end{proof}

Let $Y$ be a Riesz subspace of a Riesz space $X$. We remind the reader that $Y$ is called {\em order dense} in $X$ if, for every $0\ne x\in X_+$, there exists $y\in Y$ satisfying $0<y\le x$;
$Y$ is called {\em regular} in $X$ if $y_\alpha\downarrow 0$ in $Y$ implies $y_\alpha\downarrow 0$ in $X$. It is easy to see that order ideals and order dense Riesz subspaces 
of a Riesz space $X$ are regular in $X$. 

The following definition generalizes the concept of unbounded norm convergence with respect to an order ideal, 
which was introduced and studied in \cite{KLT}, to arbitrary  convergences.

\begin{definition}\label{uIc convergence}
Let $I$ be an order ideal in a Riesz space $X$ and $\mathbb{c}$ be a convergence on $X$. The $\mathbb{u}_I\mathbb{c}$-convergence on $X$ is defined by
$$
  x_\alpha\convuIc x \ \text{if} \ \ u\wedge|x_\alpha-x|\convc 0 \ \ \text{for} \ \ \text{all} \ \ u\in I_+.
$$
\end{definition}

\begin{theorem}\label{uc convergence}
Let $\mathbb{c}$ be a linear full convergence on $X$ and $I$ an order ideal in $X$. Then $\mathbb{u}_I\mathbb{c}$ is a linear full lattice convergence on $X$. Furthermore, 
if  $\mathbb{c}$ is also lattice, then:
\begin{enumerate} 
\item[$(i)$] \ $\mathbb{c}\subseteq\mathbb{u}_I\mathbb{c}$$;$
\item[$(ii)$] \ $\mathbb{u}_I\mathbb{u}_I\mathbb{c}=\mathbb{u}_I\mathbb{c}$$;$
\item[$(iii)$] \  $\mathbb{u}_I\mathbb{c}\in T_1$ iff $\mathbb{c}\in T_1$ and $I$ is order dense.
\end{enumerate}
\end{theorem}

\begin{proof}
Clearly $\mathbb{u}_I\mathbb{c}$ is a convergence on $X$. Firstly, we show that the convergence $\mathbb{u}_I\mathbb{c}$ is linear.
Let $(x_\alpha)_{\alpha\in A}\convuIc x$, $(y_\beta)_{\beta\in B}\convuIc y$, and $(r_\gamma)_{\gamma\in\Gamma}\to r\in\mathbb{R}$
in the standard topology on $\mathbb{R}$. Let $\gamma_0\in\Gamma$ be such that $|r_\gamma|\le|r|+1$ for all $\gamma\ge\gamma_0$.
For every $u\in X_+$, $\alpha\in A$, and $\beta\in B$,
$$
  u\wedge|x_\alpha+y_\beta-(x+y)|\le u\wedge(|x_\alpha-x|+|y_\beta-y|)\le
$$
$$
  u\wedge|x_\alpha-x|+u\wedge|y_\beta-y|.
  \eqno(9)
$$
In view of Definition \ref{linear convergence}$(a)$ and Definition \ref{top+full+lattice convergence}$(b)$, the condition $(9)$ implies
$$
   (u\wedge|x_\alpha+y_\beta-(x+y)|)_{(\alpha,\beta)\in A\times B}\convc0 \ \ \ (\forall u\in I_+),
$$
which means $(x_\alpha+y_\beta)_{(\alpha,\beta)\in A\times B}\convuIc x+y$.
Similarly, for every $u\in X_+$, $\gamma\ge\gamma_0$, and $\alpha\in A$,
$$
  u\wedge|r_\gamma x_\alpha-rx|\le u\wedge(|r_\gamma x_\alpha-r_\gamma x|+|r_\gamma x -rx|)\le
$$
$$
  u\wedge|r_\gamma||x_\alpha-x|+u\wedge|r_\gamma x -rx|\le(|r|+2)(u\wedge|x_\alpha-x|)+|r_\gamma -r||x|.
  \eqno(10)
$$
By Definition \ref{linear convergence}$(b)$ and Definition \ref{top+full+lattice convergence}$(b)$, the condition $(10)$ implies
$(u\wedge|r_\gamma x_\alpha - rx|)_{\gamma\in\Gamma;\gamma\ge\gamma_0;\alpha\in A}\convc0$ for all $u\in I_+$, 
and hence, by Definition \ref{convergence}$(c)$, $(u\wedge|r_\gamma x_\alpha - rx|)_{\gamma\in\Gamma;\alpha\in A}\convc0$,
which ensures $(r_\gamma x_\alpha)_{\gamma\in\Gamma;\alpha\in A}\convuIc rx$. Therefore, $\mathbb{u}_I\mathbb{c}$ satisfies 
Definition \ref{linear convergence} and hence is a linear convergence on $X$.

Suppose $0\le y_\alpha\le x_\alpha$ for all $\alpha\in A$
and $x_\alpha\convuIc 0$. For every $u\in I_+$, we have $u\wedge|x_\alpha|\convc 0$ and, since $\mathbb{c}$ is full
and $u\wedge|y_\alpha|\le u\wedge|x_\alpha|$ for all $\alpha\in A$, then $u\wedge|y_\alpha|\convc 0$, which means
$y_\alpha\convuIc 0$ and hence the convergence $\mathbb{u}_I\mathbb{c}$ is full.
By Proposition \ref{lattice in full}, $\mathbb{u}_I\mathbb{c}$ is a lattice convergence since $x_\alpha\convuIc 0$ trivially implies $|x_\alpha|\convuIc 0$.

$(i)$  
Since $\mathbb{c}$ is a linear full lattice convergence then $x_\alpha\convc x$ implies $u\wedge|x_\alpha-x|\convc 0$ for all $u\in I_+$ and hence $x_\alpha\convuIc x$.
Therefore $\mathbb{c}\subseteq\mathbb{u}_I\mathbb{c}$.

$(ii)$ It was proved already that $\mathbb{u}_I\mathbb{c}$ is a linear full lattice convergence. Application of $(i)$ to $\mathbb{u}_I\mathbb{c}$ provides 
$\mathbb{u}_I\mathbb{c}\subseteq\mathbb{u}_I\mathbb{u}_I\mathbb{c}$. Now, let $x_\alpha\convuIuIc x$. So
$$
    u\wedge|x_\alpha-x|\convuIc 0 \ \ \text{for} \ \ \text{all} \ \ u\in I_+,
$$
and hence
$$
     (v\wedge u)\wedge|x_\alpha-x|=v\wedge(u\wedge|x_\alpha-x|)\convc 0 \ \ \text{for} \ \ \text{all} \ \ v,u\in I_+.
$$
But the later means exactly that
$$
     w\wedge|x_\alpha-x|\convc 0 \ \ \text{for} \ \ \text{all} \ \ w\in I_+,
$$
or $x_\alpha\convuIc x$. We have shown $\mathbb{u}_I\mathbb{u}_I\mathbb{c}\subseteq\mathbb{u}_I\mathbb{c}$ and hence 
$\mathbb{u}_I\mathbb{u}_I\mathbb{c}=\mathbb{u}_I\mathbb{c}$ as desired.

$(iii)$ 
Let $\mathbb{c}\in T_1$ and the order ideal $I$ be order dense in $X$. Since  $\mathbb{u}_I\mathbb{c}$ is linear, by Lemma \ref{T_1 for linear conv}, 
for proving $\mathbb{u}_I\mathbb{c}\in T_1$, we have to show that $X\ni a\convuIc b$ implies $a = b$.
Let $a\convuIc b$. Then $u\wedge|a|=u\wedge|a-b|\convc 0$  for every $u\in I_+$. Since $\mathbb{c}\in T_1$ then $u\wedge|a-b|=0$  for every $u\in I_+$, and,  
since $I$ is order dense, $a=b$.  Therefore we conclude $\mathbb{u}_I\mathbb{c}\in T_1$.

Let now $\mathbb{u}_I\mathbb{c}\in T_1$. By $(i)$,  $\mathbb{c}\subseteq\mathbb{u}_I\mathbb{c}\in T_1$, and hence $\mathbb{c}\in T_1$.
In order to show that $I$ is order dense, it is enough to show that if, for some $x\in X$,  $u \wedge | x | = 0$ for all $u \in I_+$, then $x=0$.
So, let  $x\in X$ and
$$
           u\wedge|x-0| = 0 \ \ \text{for} \ \ \text{all} \ \ u\in I_+.
$$
Hence
$$
           u\wedge|x-0| \convc 0 \ \ \text{for} \ \ \text{all} \ \ u\in I_+, 
$$
and by Definition \ref{uIc convergence}, $x \convuIc 0$. Since $\mathbb{u}_I\mathbb{c}\in T_1$, it follows $x=0$ as desired.
\end{proof}

If $I=X$, $\mathbb{u}_I\mathbb{c}$ is denoted by $\mathbb{u}\mathbb{c}$. By taking $I=X$, the next corollary follows from Theorem \ref{uc convergence}.

\begin{corollary}\label{T_1 uc} 
Let $\mathbb{c}$ be a linear full lattice convergence on $X$. Then $\mathbb{c}\in T_1$ iff $\mathbb{uc}\in T_1$.
\end{corollary}

By Proposition \ref{o-convergence is full lattice and T_1}, the $\mathbb{o}$-convergence and $\mathbb{p}$-convergence on an  Archi\-me\-de\-an Riesz space
are linear full lattice $T_1$ convergences. The ${\text{\boldmath{$\tau$}}}$-convergence on a Hausdorff locally solid Riesz space is  also a linear full lattice $T_1$ convergence.
Therefore, the following proposition generalizes \cite[Prop.3.15]{GTX}, \cite[Prop.3.2]{AEEM1}, and \cite[Prop.2.12]{Tay2}.

\begin{proposition}\label{Convergence and sublattices}
Let $\mathbb{c}$ be a linear full lattice $T_1$ convergence on an  Archi\-me\-de\-an  Riesz space $X$ such that $\mathbb{o}\subseteq\mathbb{c}$. 
Let $Y$ be a Riesz subspace of $X$. Then $Y$ is $\mathbb{uc}$-closed in $X$ iff $Y$ is $\mathbb{c}$-closed in $X$.   
\end{proposition}

\begin{proof}
First of all, notice that $\mathbb{uc}$ is linear, full, lattice, and $T_1$ by Theorem~\ref{uc convergence}. 

For the necessity, let  $Y$ be $\mathbb{uc}$-closed in $X$, and let  $Y\ni y_\alpha\convc x\in X$. Since $\mathbb{c}$ is a linear full lattice convergence,
$u\wedge|y_\alpha-x|\convc 0$ for all $u\in X_+$, and hence $y_\alpha\convuc x$, showing $x\in Y$ because  $Y$ is $\mathbb{uc}$-closed in $X$. 

For the sufficiency, let $Y$ be $\mathbb{c}$-closed in $X$ and $Y\ni y_\alpha\convuc x\in X$. 
By Theorem \ref{double nets}, without loss of generality, we assume $y_\alpha\in Y_+$ for all $\alpha$. Since $y_\alpha\convuc x$ then
$z \wedge |y_\alpha - x | \convc 0$ for all $z \in X_+$, and hence, since $\mathbb{c}$ is linear and full, the inequality 
$|z \wedge y_\alpha - z \wedge x | \le  z \wedge |y_\alpha - x |$ implies $z \wedge y_\alpha \convc z \wedge x$.
In particular, $Y_+ \ni y \wedge y_\alpha \convc y \wedge x$ for any $y\in Y_+$,  and since $Y$  is $\mathbb{c}$-closed in $X$ then
$y \wedge x \in Y$ for each $y\in Y_+$. 

Observe that, for all $z\in Y^d_+$,  $0=y_\alpha\wedge z\convuc x\wedge z$ by Theorem \ref{double nets}. In view of $\mathbb{uc}\in T_1$, 
by Lemma \ref{T_1 for linear conv} we get $x\wedge z=0$ for all $z\in Y^d_+$, and hence $x\in Y^{dd}$. 
Since $Y^{dd}$ is the band generated by $Y$ in $X$, there exists a net $z_\beta$ in the order ideal $I_Y$ generated by $Y$ such that $0\le z_\beta\uparrow x$. 
Let $z_\beta\le w_\beta\in Y_+$ for every $\beta$. It follows from $x\ge w_\beta\wedge x\ge z_\beta\wedge x=z_\beta\uparrow x$
that $w_\beta\wedge x\convo x$ and hence $w_\beta\wedge x\convc x$  because  $\mathbb{o}\subseteq\mathbb{c}$. 
Thus, $w_\beta \wedge x \in Y$ for all $\beta$. The obtained together with $\mathbb{c}$-closeness of $Y$  imply  $x\in Y$, as required. 
\end{proof}

\begin{remark}\label{utau convergence}
Let $I$ an order ideal in a Riesz space $X$ equipped with linear locally full topology $\tau$.  
By Theorem $\ref{uc convergence}$, the $\mathbb{u}_I\mathbb{\tau}$-convergence on $X$ is linear full and lattice.
Let $\mathcal{T}$ be the family of all full $\tau$-neighborhoods of zero. The following observations are straightforward.
\begin{enumerate} 
\item[$(i)$] \ 
Let $B:=\{U_{u,V}\}_{u\in I_+, V\in\mathcal{T}}$, where 
$$
  U_{u,V}:=\{x\in X: u\wedge|x|\in V\} \ \ \ \ (u\in I_+, V\in\mathcal{T}).
$$
Then $B$ is a base at zero of some linear topology on $X$. To see this, let $u_1,u_2\in I_+$, $V_1,V_2\in\mathcal{T}$. 
Pick a full $V_3\in\mathcal{T}$ such that $V_3\subseteq V_1\cap V_2$. Then 
$$
   U_{u_1\vee u_2, V_3}\subseteq U_{u_1,V_3}\bigcap U_{u_2,V_3}\subseteq U_{u_1,V_1}\bigcap U_{u_2,V_2}.
$$
Since each $V\in\mathcal{T}$ is full, for every $x\in X$ there exists small enough $\varepsilon>0$ 
with $u\wedge|tx|\le|tx|\in V$ for all real  $t$, $|t|\le\varepsilon$. This shows that every set $U_{u,V}\in B$ is absorbing.
For proving that $B$ consists of balanced sets take $x$ in some $U_{u,V}\in B$ and let $|t| \le 1$. 
Since $0\le u\wedge|t x| \le u\wedge|x| \in V$ and $V$ is full, then $tx \in U_{u,V}$. 
So, every  $U_{u,V}\in B$ is balanced. It remains to show that for each 
$U_{u,V}\in B$ there exists $U_{u,W}\in B$ with $U_{u,W}+U_{u,W} \subseteq U_{u,V}$.
Take arbitrary $U_{u,V}\in B$. Since $\mathcal{T}$ is a zero base for topology $\tau$,
there exists $W\in\mathcal{T}$ with $W  + W \subseteq V$. If $x, y \in U_{u,W}$ then $u\wedge|x|,u\wedge|y|\in W$. 
Since $V$ is full, it follows from
$$
   u\wedge|x+y| \le u\wedge (|x|+|y|) \le  u\wedge|x| + u\wedge|y| \in W + W \subseteq V
$$ 
that $x+y \in U_{u,V}$. Therefore, $U_{u,W} + U_{u,W}  \subseteq U_{u,V}$, and hence 
$B$ is a zero base of a linear topology, which will be referred to as the topology $u\tau_I$ on $X$.
\item[$(ii)$] \ The $\mathbb{u}_I\mathbb{\tau}$-convergence on $X$ is  topological with respect to $u\tau_I$.
\item[$(iii)$] \ The linear topology $u\tau_I$ is locally solid. Indeed, take $W\in u\tau_I(0)$. Let $U_{u,V}\subseteq W$. 
Clearly $U_{u,V} \in u\tau_I(0)$. Next, we claim that $U_{u,V}\in B$ is solid. 
To see this, take $x\in U_{u,V}$ and let $|z|\le|x|$. Since $V$ is full then 
$$
  0\le u\wedge|z|\le u\wedge|x|\in V \Longrightarrow u\wedge|z|\in V.
$$ 
Thus $z\in U_{u,V}$, which shows that $U_{u,V}$ is solid. Therefore $W$ contains the solid $u\tau_I$-neighborhood $U_{u,V}$ of zero, 
and hence  the topology $u\tau_I$ is locally solid.  
\item[$(iv)$] \ It should be clear from $(i)$, $(ii)$, and $(iii)$, that the crucial ingredient in construction of ``unbounded topologies" is
local fullness of the original topology.
\item[$(v)$] \ In the recent literature $($see, e.g., $\cite{DEM2,EEG,Tay1,Tay2,AEEM1}$$)$, most of results related to unbounded topologies 
are stated for a locally solid Riesz space $(X,\tau)$. It seems to be meaningful to analyze these results in the reverse direction by relaxing
local solidness of the topology $\tau$ on a Riesz space $X$ to local fullness of $\tau$.   
\end{enumerate}
\end{remark}

\begin{definition}\label{min max convergences}
Let $C$ be a class of full $T_1$ convergences on a Riesz space $X$. A convergence $\mathbb{c}\in C$ is called$:$
\begin{enumerate} 
\item[$(a)$] \ {\em minimal in $C$} if, for any  $\mathbb{c}_1\in C$,
$$
    \mathbb{c}_1\subseteq\mathbb{c}\ \Longrightarrow \ \mathbb{c}_1=\mathbb{c};
$$
\item[$(b)$] \ {\em maximal in $C$} if, for any $\mathbb{c}_1\in C$,
$$
    \mathbb{c}\subseteq\mathbb{c}_1\ \Longrightarrow \ \mathbb{c}_1=\mathbb{c}.
$$
\end{enumerate}
\end{definition}

\begin{remark}\label{minimal}
Let $X$ be a Riesz space.
\begin{enumerate} 
\item[$(i)$] \ The full lattice convergence $\mathbb{c_d}$ from Example $\ref{c_a and c_d}$ is minimal in the class of all full $T_1$ convergences on $X$.
\item[$(ii)$] \ Let $I$ be an order dense order ideal in $X$ and $\mathbb{c}$ be  maximal in the class of all linear full lattice $T_1$ convergences on $X$. 
Then $\mathbb{u}_I\mathbb{c}=\mathbb{c}$, by Theorem $\ref{uc convergence}$.
\item[$(iii)$] \ By Remark $\ref{utau convergence}$$(ii)$, if $\mathbb{c}$ is the ${\text{\boldmath{$\tau$}}}$-convergence on a locally full Riesz space $(X,\tau)$, 
then the full lattice convergence $\mathbb{u}_I\mathbb{c}$ 
is also topological with respect to the locally solid topology $u\tau_I\subseteq\tau$ on $X$ $($cf. also $\cite{DEM2,Tay2}$$)$.  
\end{enumerate} 
\end{remark}

While the existence of a minimal full lattice $T_1$ convergence is trivial, the existence of a maximal full lattice $T_1$ convergence is rather deep and interesting issue,
e.g., in the class of all topological convergences with respect to locally full $T_1$ topologies on a Riesz space (see, for example \cite{Tay1,Tay2}).

It is worth mention that although the $\mathbb{o}$-convergence on a Riesz $X$ is not topological unless $\dim(X)<\infty$, 
the full lattice $\mathbb{uo}$-convergence on every discrete Riesz space $X$ is always topological \cite[Thm.2]{DEM3} with
respect to a locally solid topology on $X$.

\section{$\mathbb{mc}$-Convergence on Commutative $l$-Algebras}

A Riesz space $X$ is called an {\em $l$-algebra}, if $X$ is also an associative algebra whose positive cone $X_+$ is closed under the algebra multiplication. 
We recall the following definitions (cf. \cite{Pag,BH,Hu,BR,BC}).

\begin{definition}\label{various lattice algebras}
An $l$-algebra $X$ is called
\begin{enumerate}
\item[$(a)$] \ {\em $d$-algebra} whenever the multiplication is distributive with respect to the lattice operations, i.e., for all $u,x,y\in X_+$,
$$
  u\cdot(x\wedge y)=(u\cdot x)\wedge(u\cdot y) \ \ \text{and} \ \ (x\wedge y)\cdot u=(x\cdot u)\wedge(y\cdot u);
$$
\item[$(b)$] \ {\em almost $f$-algebra} if, for all $x,y\in X_+$, $x\wedge y=0\Longrightarrow x\cdot y=0$$;$
\item[$(c)$] \ {\em $f$-algebra} if, for all $u,x,y\in X_+$,
$$
  x\wedge y=0 \Longrightarrow \ (u\cdot x)\wedge y=(x\cdot u)\wedge y=0;
$$
\item[$(d)$] \ {\em semiprime} whenever the only nilpotent element in $X$ is $0$$;$
\item[$(e)$] \ {\em unital} if $X$ has a multiplicative unit.
\end{enumerate}
\end{definition}

\begin{example}\label{l-algebra $R$}
In the real field $\mathbb{R}$, any $l$-algebra multiplication $*$ is uniquely determined by $1*1$.
Indeed, for any $a,b\in\mathbb{R}$, $a*b=(a\cdot 1)*(b\cdot 1)=a\cdot b\cdot (1*1)$. In particular, any
one-dimensional $l$-algebra is commutative.
\end{example}

\begin{example}\label{l-algebra $R^D$}
In the Riesz space $\mathbb{R}^D$ of all $\mathbb{R}$-valued functions on a set $D$, any $f$-algebra multiplication $*$ is uniquely determined by
the function $\zeta(d):=[\mathbb{e}^d*\mathbb{e}^d](d)$, where $\mathbb{e}^d\in\mathbb{R}^D$ equals $1$ at $d$ and $0$ otherwise. 
Indeed, for any $a, b \in\mathbb{R}^D$, 
$$
  [a * b](d)=a(d)\mathbb{e}^d*b(d)\mathbb{e}^d=a(d)b(d)(\mathbb{e}^d*\mathbb{e}^d)=\zeta(d)a(d)b(d).
$$
Therefore, $a * b=\zeta\cdot a\cdot b$, where  $\zeta\cdot a\cdot b$ is the pointwise product of the functions $\zeta$, $a$, and $b$ from  $\mathbb{R}^D$.
In particular, $\mathbb{R}^D$ with the multiplication $*$ is unital iff $\mathbb{R}^D$ 
is semiprime iff $\zeta$ is a weak order unit in $\mathbb{R}^D$.
\end{example}

\begin{definition}\label{l-ideal}
A subspace $L$ of an $l$-algebra $X$ is called an {\em $l$-ideal} whenever $L$ is an order ideal in the Riesz space $X$ and a ring ideal in the algebra $X$.
\end{definition}

\begin{remark}\label{about l-ideals}
If $X$ is an Archimedean $f$-algebra then any $\mathbb{r}$-closed order ideal in $X$ is an $l$-ideal $($see $\cite[Thm.10.5]{Pag}$$)$.
\end{remark}

\begin{definition}\label{c-cont multiplication}
Let $\mathbb{c}$ be a linear convergence on a commutative $l$-algebra $X$. 
The algebra multiplication in $X$ is called {\em $\mathbb{c}$-continuous} whenever
$$ 
  x_\alpha\convc x \ \Longrightarrow \  y\cdot x_\alpha\convc y\cdot x
$$ 
for every net $x_\alpha$ in $X$ and all $x,y\in X$.
\end{definition}

\begin{example}\label{almost f-algebra in which multiplication is not 0-cont}
Let $\cal{U}$ be a free ultrafilter on $\mathbb{N}$. Recall that a real sequence $x_n$
{\em converges along} $\cal{U}$ to $x_0\in\mathbb{R}$ whenever $\{k\in\mathbb{N}: |x_k-x_0|\le\varepsilon\}\in\cal{U}$
for every $\varepsilon>0$. It is well known that any sequence $x=(x_n)_{n=1}^\infty\in\ell_\infty$ 
converges along $\cal{U}$ to its limit $x_{\cal{U}}:=\lim_{\cal{U}}x_n$. 

We define an operation $*$ in $\ell_\infty$ by $x*y:=(\lim_{\cal{U}}x_n)\cdot(\lim_{\cal{U}}y_n)\cdot\mathbb{1}$,
where $\mathbb{1}$ is a sequence of reals identically equal to $1$. It is straightforward to check that$:$
\begin{enumerate}
\item[$(a)$] \ $x*y:=(\lim_{\cal{U}}x_n\cdot y_n)\cdot\mathbb{1}$$;$
\item[$(b)$] \ $*$ is a commutative algebra multiplication in $\ell_\infty$$;$
\item[$(c)$] \ $(\ell_\infty,*)$ is an Archimedean almost $f$-algebra$;$
\item[$(d)$] \ $(\ell_\infty,*)$ is a $d$-algebra$;$
\item[$(e)$] \ $(\ell_\infty,*)$ is neither semiprime nor unital$;$
\item[$(f)$] \ $(\ell_\infty,*)$ is not an $f$-algebra$;$
\item[$(g)$] \ The algebra multiplication $*$ in the Banach lattice $\ell_\infty$ is norm continuous.
\end{enumerate}

The algebra multiplication $*$ in $\ell_\infty$ is not $\mathbb{o}$-continuous. To see this, for every $n\in\mathbb{N}$, take    
the indicator functions of $\{k\in\mathbb{N}:k\ge n\}$ denoting  $\mathbb{1}_n:=\mathbb{1}_{\{k\in\mathbb{N}:k\ge n\}}\in\ell_\infty$.
Clearly $\mathbb{1}_n\convo 0$ in $\ell_\infty$, yet $\mathbb{1}*\mathbb{1}_n=\mathbb{1}$ for 
all $n\in\mathbb{N}$.

Since $\mathbb{1}_n\convuo 0$, we also conclude that the multiplication $*$ in $\ell_\infty$ is not $\mathbb{uo}$-continuous. 
\end{example}

Let $X$ be an Archimedean Riesz space. Recall that an order bounded linear operator $\pi$ in $X$ is called an {\em orthomorphism}, 
if $|x|\wedge|y|=0$ in $X$ implies $|\pi x|\wedge|y|=0$. The set of all orthomorphisms in $X$ is denoted $\text{Orth}(X)$. By 
$\cite[Thms.9.3, 9.4, 15.1]{Pag}$ and $\cite[Thms.140.9, 140.10]{Za}$, $\text{Orth}(X)$ endowed with the operator 
multiplication is a commutative semi\-prime $f$-algebra in which the identity operator $I$ is a multiplicative unit $($cf. $\cite[Thm.4.1]{Hu}$$)$. 
Since any orthomorphism is $\mathbb{o}$-continuous $\cite[Thm.8.6]{Pag}$, the algebra multiplication in $\text{Orth}(X)$ is $\mathbb{o}$-continuous.
Furthermore, if $X$ is Dedekind complete, $\text{Orth}(X)$ coincides with the band generated by $I$ in $L_b(X)$. In particular, for a Dedekind complete $X$, 
$\text{Orth}(X)$ is a Dedekind complete unital $f$-algebra $\cite[Thm.15.4]{Pag}$.

\begin{remark}\label{various properties of l-algebras}
The following well known facts are used freely in the present paper.
\begin{enumerate}
\item[$(i)$] \ An $l$-algebra $X$ is semiprime iff $x^2=0\Rightarrow x=0$ in $X$.
\item[$(ii)$] \ An $f$-algebra $X$ is semiprime iff $x\cdot y=0\Rightarrow |x|\wedge|y|=0$ in $X$.
\item[$(iii)$] \ Any $f$-algebra is both $d$- and almost $f$-algebra.
\item[$(iv)$] \ Any Archimedean $f$-algebra $X$ with a multiplicative unit $e$ is semiprime, and
$e$ is a weak order unit in $X$.
\item[$(v)$] \ If $X$ is an Archimedean semiprime $l$-algebra then$:$
$$
  \text{X is an f-algebra iff X is an almost f-algebra iff X is a d-algebra}.
$$
\item[$(vi)$] \ If $X$ is an Archimedean $l$-algebra with a positive multiplicative unit then$:$
$$
  \text{X is an f-algebra iff  X is an almost f-algebra iff X is a d-algebra}.
$$
\item[$(vii)$] \ If $X$ is an Archimedean $f$-algebra then the algebra multiplication in $X$ is commutative and order continuous $($see $\cite{Pag,Hu}$$)$.
\item[$(viii)$] \ An Archimedean $f$-algebra $X$ is Riesz and algebra isomorphic to the $f$-algebra $\text{Orth}(X)$ iff $X$ is unital $\cite[Thm.4.3]{Hu}$.
Furthermore, the isomorphism $X\ni x\to\pi(x)\in\text{Orth}(X)$ is the multiplication by $x$, i.e.  $\pi(x)(y):=x\cdot y$.
\item[$(ix)$] \ In any $d$-algebra $X$, $|x\cdot y|=|x|\cdot|y|$ for all $x,y\in X$.
\end{enumerate}
\end{remark}

As Example \ref{s with noncontinuous multiplication} below shows, the algebra multiplication in a universally complete $f$-algebra 
need not to be $\mathbb{c}$-continuous with respect to a full lattice convergence. However, for some modes 
of full lattice convergences, the algebra multiplication in an arbitrary commutative $l$-algebra is always continuous.

\begin{proposition}
The algebra multiplication in any commutative $l$-algebra $X$ is $\mathbb{r}$-continuous.
\end{proposition}

\begin{proof}
Let $x_\alpha\convr x$ in $X$ and $y\in X$. The convergence $x_\alpha\convr x$ implies the existence of $v\in X_+$ such that, for any $\varepsilon>0$, 
there is $\alpha_\varepsilon$ with $|x_\alpha-x|\le\varepsilon v$ for all $\alpha\ge\alpha_\varepsilon$. Take $u:=|y|\cdot v$. 
Then the following relations
$$
  |y\cdot x_\alpha- y\cdot x|\le|y|\cdot|x_\alpha-x|\le\varepsilon|y|\cdot v=\varepsilon u \ \ \ \ (\forall \alpha\ge\alpha_\varepsilon)
$$ 
imply  $y\cdot x_\alpha\convr y\cdot x$ in $X$, as required.
\end{proof}

\begin{lemma}\label{lattice convergence and c-cont multiplication}
Let $\mathbb{c}$ be a linear lattice convergence on a commutative $l$-algebra $X$. 
The multiplication in $X$ is $\mathbb{c}$-continuous iff 
$$ 
  x_\alpha\convc 0 \ \Rightarrow \  y\cdot x_\alpha\convc 0
$$
for every net $x_\alpha$ in $X_+$ and every $y\in X_+$. 
\end{lemma}

\begin{proof}
Only the sufficiency requires a proof. Let $x_\alpha\convc x$ in $X$ and $y\in X$.
Since $\mathbb{c}$ is a linear lattice convergence, we get, by the assumption,
$$
  y^+\cdot(x_\alpha-x)^{\pm}\convc 0 \ \ \text{and} \ \ y^-\cdot(x_\alpha-x)^{\pm}\convc 0.
$$
Then
$$
  y\cdot(x_\alpha-x)=(y^+-y^-)\cdot((x_\alpha-x)^+-(x_\alpha-x)^-)=
$$
$$	
	y^+(x_\alpha-x)^+ - y^+(x_\alpha-x)^-- 
	y^-(x_\alpha-x)^++y^-(x_\alpha-x)^-\convc 0.
$$
Therefore, $y\cdot x_\alpha\convc y\cdot x$, as required.
\end{proof}

\begin{lemma}\label{product convergence}
Let $\mathbb{c}$ be a linear full lattice convergence on a commutative $l$-algebra $X$ such that the multiplication in $X$ is $\mathbb{c}$-continuous.
If $(x_\alpha)_{\alpha\in A}\convc x$, $(y_\beta)_{\beta\in B}\convc y$, and at least one of these two nets is eventually order bounded, then 
$(x_\alpha\cdot y_\beta)_{(\alpha,\beta)\in A\times B}\convc x\cdot y$.
\end{lemma}

\begin{proof}
Without loss of generality, suppose $|x_\alpha|\le v$ for all $\alpha\ge\alpha_0$. Then, for all $\alpha\in A$, $\alpha\ge\alpha_0$, $\beta\in B$,
it holds that
$$
   |x_\alpha\cdot y_\beta-x\cdot y|=|x_\alpha\cdot y_\beta-x_\alpha\cdot  y+x_\alpha\cdot  y-x\cdot y|\le
$$
$$	
   |x_\alpha|\cdot|y_\beta-y|+|y|\cdot|x_\alpha-x|\le
   v\cdot|y_\beta-y|+|y|\cdot|x_\alpha-x|.
   \eqno(11)
$$
Pick any $\beta_0\in B$. Using that the multiplication in $X$ is $\mathbb{c}$-continuous, $\mathbb{c}$ is a linear lattice convergence, and Definition \ref{convergence}$(b)$,
we obtain 
$$
  (v\cdot|y_\beta-y|)_{\beta\in B;\beta\ge\beta_0}\convc 0 \ \ \text{and} \ \ (|y|\cdot|x_\alpha-x|)_{\alpha\in A;\alpha\ge\alpha_0}\convc 0.
$$	
Since the convergence $\mathbb{c}$ is full, $(11)$ implies
$$
    (|x_\alpha\cdot y_\beta-x\cdot y|)_{(\alpha,\beta)\in A\times B;\alpha\ge\alpha_0;\beta\ge\beta_0}\convc 0.
$$
Thus, we get $(x_\alpha\cdot y_\beta)_{(\alpha,\beta)\in A\times B;\alpha\ge\alpha_0;\beta\ge\beta_0}\convc x\cdot y$. By
Definition \ref{convergence}$(c)$, $(x_\alpha\cdot y_\beta)_{(\alpha,\beta)\in A\times B}\convc x\cdot y$, as required.
\end{proof}

The following definition introduces a modification of a linear convergence on a commutative $l$-algebra by using its algebra multiplication. 
This modification has significant differences and some minor similarities with the $\mathbb{u}_I\mathbb{c}$-convergence, which was well studied 
recently for several full lattice convergences (see \cite{GTX,AEEM1,Tay2} and the references therein).

\begin{definition}\label{mc convergence}
Let $\mathbb{c}$ be a convergence on a commutative $l$-algebra $X$, and $x_\alpha$ be a net in $X$. 
The net $x_\alpha$ {\em multiplicative} $\mathbb{c}$-converges to $x$ $($briefly $x_\alpha\convmc{x}$$)$ if 
$$
  u\cdot|x_\alpha-x|\convc{0} \ \ \ \ (\forall u\in X_+).
  \eqno(12)
$$	
\end{definition}

It is immediate to see that $\mathbb{mc}$ satisfies conditions $(a)$, $(b)$, and $(c)$ of Definition \ref{convergence}.
Therefore, $\mathbb{mc}$ is a convergence on $X$. As it will be shown below, in order to get that $\mathbb{mc}$
is linear some further conditions on $\mathbb{c}$ are required. In the cases when $\mathbb{c}$ is either the 
order convergence or the norm convergence, the notion of correspondent $\mathbb{mc}$-convergence was introduced 
and studied by the first author in \cite{Ay1,Ay2}.

It seems to make sense to study $(12)$ by using single (or repeated) left, right, both sides, etc., algebra multiplication in arbitrary $l$-algebras. 
However, in order to simplify the presentation, in Definition \ref{mc convergence}, we restrict ourselves to single multiplication in a commutative $l$-algebra. 

In the rest of the section, we study general properties of $\mathbb{mc}$-conver\-gence.

\begin{lemma}\label{mc net of nilpotents}
Let $\mathbb{c}$ be an additive $T_1$-convergence on an Archimedean $f$-algebra. 
Then the set $N(X)$ of all nilpotent elements in $X$ is $\mathbb{mc}$-closed.
\end{lemma}

\begin{proof}
We shall use in the proof the facts that every $f$-algebra is a $d$-algebra, and that $|x\cdot y|=|x|\cdot|y|$ in any $d$-algebra 
(cf. $(iii)$ and $(ix)$ of Remark \ref{various properties of l-algebras}).
Let $N(X)\ni x_\alpha\convmc x\in X$. Then, for every $u\in X_+$,
$$
  u\cdot|x_\alpha-x|=|u\cdot(x_\alpha-x)|=|u\cdot x_\alpha-u\cdot x| = |u\cdot x|\convc 0,
$$
where the final equality following from  \cite[Prop.10.2]{Pag}.
Since $\mathbb{c}\in T_1$, Lemma \ref{T_1 for linear conv} implies $|u\cdot x|=0$ and hence $u\cdot x=0$ for every $u\in X_+$. 
Then $y\cdot x=0$ for all $y\in X$. In particular, $x\cdot x=0$. Hence $x\in N(X)$.
\end{proof}

\begin{lemma}\label{mc implies c}
Let $\mathbb{c}$ be an additive full convergence on a commutative $l$-algebra $X$ with a positive algebraic unit.
Then $\mathbb{mc}\subseteq\mathbb{c}$.
\end{lemma}

\begin{proof} 
Let $x_\alpha\convmc x$. Then $u\cdot|x_\alpha-x|\convc0$ for all $u\in X_+$. By taking $u$ a positive algebraic unit in $X$, 
$|x_\alpha-x|\convc0$. Since $\mathbb{c}$ is full and $|x_\alpha-x|\ge(x_\alpha-x)^{\pm}\ge 0$, it follows that $(x_\alpha-x)^{\pm}\convc 0$.
Since  $\mathbb{c}$ is additive,  $x_\alpha-x=(x_\alpha-x)^+-(x_\alpha-x)^-\convc 0$ and hence $x_\alpha\convc x$.
\end{proof}

\begin{example}\label{mw convergence}
Consider the  Banach $f$-algebra $L_\infty[0,1]$ with  pointwise multiplication. Clearly the function $e(t)\equiv 1$ 
is both a multiplicative and order unit in $L_\infty[0,1]$. By Example $\ref{full but not lattice}$, the weak convergence $\mathbb{w}$ 
on $L_\infty[0,1]$ is a  full yet non-lattice convergence. Since $X$ has positive algebraic unit, Lemma $\ref{mc implies c}$
implies $\mathbb{mw}\subseteq\mathbb{w}$.
\end{example}

The next two examples emphasize the importance of the assumption that $X$ has positive algebraic unit in Lemma \ref{mc implies c}.

\begin{example}\label{0}
Let $X$ by a non-Dedekind complete Riesz space. Let $\circ$ be the trivial algebra multiplication in $X$$:$
$x\circ y=0$ for all $x,y\in X$. The following observations are straightforward$:$
\begin{enumerate}
\item[$(a)$] \ $(X,\circ)$ is a non-unital $\mathbb{mo}$-complete $f$-algebra with $\mathbb{o}$-continuous algebra multiplication$;$
\item[$(b)$] \ $x_\alpha\convmo 0$ yields for any $\mathbb{o}$-divergent $\mathbb{o}$-Cauchy net $x_\alpha$ in $(X,\circ)$,
in particular $\mathbb{mo}\not\subseteq\mathbb{o}$.
\end{enumerate}
\end{example}

\begin{example}\label{c_00}
Consider the Dedekind complete non-unital $f$-algebra $c_{00}$ of all eventually zero real sequences with coordinatewise ordering and algebra multiplication.
Let $\mathbb{c}$ be the $\|\cdot\|_{\infty}$-convergence on $c_{00}\subset\ell_{\infty}$. Clearly$:$ $\mathbb{c}$ is a full lattice convergence$;$  
the algebra multiplication in $c_{00}$ is $\mathbb{c}$-continuous$;$ $\mathbb{c}\subseteq\mathbb{mc}$ in $c_{00}$ but not vice verse, 
$e_n\convmc 0$ in $c_{00}$ yet the sequence $e_n$ does not $\mathbb{c}$-converge.
\end{example}

The following example shows that  the statement of Lemma \ref{mc implies c} holds true also for some non-full convergences.

\begin{example}\label{PP[0,1]}
The Riesz space $PP[0,1]$ from Example $\ref{piece wise polynomial}$ is a commutative $f$-algebra under the pointwise multiplication $\cite[Ex.5.1(i); Ex.5.8]{Hu}$. 
The function $e(t)\equiv 1$ on $[0,1]$ is an order and multiplicative unit in $PP[0,1]$. The convergence $\mathbb{c_1}$ on  $PP[0,1]$ as in Example $\ref{piece wise polynomial}$ 
is not full. It is easy to see that $\mathbb{mc_1}=\mathbb{c_1}$.
\end{example}

\begin{lemma}\label{c implies mc}
Let $\mathbb{c}$ be an additive lattice convergence on a commutative $l$-algebra $X$ with a $\mathbb{c}$-continuous algebra multiplication.
Then $\mathbb{c}\subseteq\mathbb{mc}$.
\end{lemma}

\begin{proof}
Let $x_\alpha\convc x$. Since $\mathbb{c}$ is additive, $x_\alpha-x \convc0$; and since $\mathbb{c}$ is lattice, $|x_\alpha-x|\convc0$.
By  $\mathbb{c}$-continuity of the multiplication, $u\cdot|x_\alpha-x|\convc0$ for all $u\in X_+$, and hence $x_\alpha\convmc x$.
\end{proof}

\begin{theorem}\label{c=mc}
Let $\mathbb{c}$ be an additive full lattice convergence on a commutative $l$-algebra $X$ with a positive algebraic unit and $\mathbb{c}$-continuous algebra multiplication.
Then the $\mathbb{mc}$-convergence coincides with the $\mathbb{c}$-convergence on $X$.
\end{theorem}

\begin{proof} 
It follows from Lemma \ref{mc implies c} and Lemma \ref{c implies mc}.
\end{proof}

In general, the assumption of $\mathbb{c}$-continuity of the multiplication in $X$ can not be drooped in Theorem \ref{c=mc}.

\begin{example}\label{s with noncontinuous multiplication}
Consider the unital $f$-algebra $s$ of all real sequences with the coordinatewise ordering and algebra multiplication.
Let $\mathbb{c}$ be the eventually $\|\cdot\|_{\infty}$-bounded coordinatewise convergence, that is$:$ 
$((x_{n,\beta})_n^\infty)_\beta\convc 0$ in $X$ whenever $(x_{n,\beta})_n^\infty$ $\beta$-eventually lies in $\ell_\infty$ and
$(x_{n,\beta})_\beta\to 0$ for every $n\in\mathbb{N}$. Then $\mathbb{mc}\subseteq\mathbb{c}$ but not vice verse: the sequence $e_n$ 
$\mathbb{c}$-converges to $0$ yet $e_n$ does not $\mathbb{mc}$-converge anywhere. It should be noticed that the convergence $\mathbb{c}$ is not $\mathbb{o}$-continuous. 
Furthermore, the algebra multiplication in $s$ is not $\mathbb{c}$-continuous. Indeed, $e_n\convc 0$ yet $(n)_{n=1}^{\infty}\cdot e_n=ne_n$ 
does not $\mathbb{c}$-converge to $0$ since it is not eventually $\|\cdot\|_{\infty}$-bounded.  
\end{example}

\begin{theorem}\label{full lattice convergence to mc}
If $\mathbb{c}$ is a linear  full convergence on a commutative $l$-algebra $X$, then $\mathbb{mc}$ is a linear full lattice convergence on $X$.
\end{theorem}

\begin{proof} 
Firstly, we show that $\mathbb{mc}$ is linear. Let 
$(x_\alpha)_{\alpha\in A}\convmc x\in X$, $(y_\beta)_{\beta\in B}\convmc y\in X$,  and $(r_\gamma)_{\gamma\in\Gamma}\to r$
in the standard topology on $\mathbb{R}$. Let $\gamma_0\in\Gamma$ be such that $|r_\gamma|\le|r|+1$ for all $\gamma\ge\gamma_0$.
Since, for all $u\in X_+$, $\alpha\in A$, and $\beta\in B$, we have
$$
  u\cdot|x_\alpha+y_\beta-(x+y)|\le u\cdot|x_\alpha-x|+u\cdot|y_\beta-y|,
  \eqno(13)
$$
and since $x_\alpha\convmc x$, $y_\beta\convmc y$, then $u\cdot|x_\alpha-x|+u\cdot|y_\beta-y|\convc 0$ 
for all $u\in X_+$. Since $\mathbb{c}$ is full, $(13)$ implies 
$$
  (x_\alpha+y_\beta-(x+y))_{(\alpha,\beta)\in A\times B}\convmc 0,
$$ 
and hence $(x_\alpha+y_\beta)_{(\alpha,\beta)\in A\times B}\convmc x+y$, which proves the $\mathbb{mc}$-continuity of the addition in $X$. 

Pick any $\alpha_0\in A$ and remark that, for all $u\in X_+$, $\gamma\ge\gamma_0$, and $\alpha\ge\alpha_0$, we have
$$
  u\cdot|r_\gamma x_\alpha-rx|\le u\cdot(|r_\gamma x_\alpha-r_\gamma x|+|r_\gamma x -rx|)=
$$
$$
  u\cdot|r_\gamma||x_\alpha-x|+u\cdot|r_\gamma x -rx|\le
  (|r|+1)u\cdot|x_\alpha-x|+|r_\gamma -r|u\cdot|x|.
  \eqno(14)
$$
Since $x_\alpha\convmc x$ and $r_\gamma\to r$, 
$$
  (|r|+1)u\cdot|x_\alpha-x|+|r_n-r|u\cdot|x|\convc 0 \ \ \ (\forall u\in X_+) \  .
  \eqno(15)
$$
Since $\mathbb{c}$ is full, $(14)$ and $(15)$ imply  
$$
  (u\cdot|r_\gamma x_\alpha-rx|)_{(\gamma_0,\alpha_0)\le(\gamma,\alpha)\in\Gamma\times A}\convc 0 \ \ \ (\forall u\in X_+,\gamma\ge\gamma_0).
$$
By the condition $(c)$ of Definition \ref{convergence}, $u\cdot|r_\gamma x_\alpha-rx|\convc 0$ and hence
$(r_\gamma x_\alpha-rx)_{(\gamma,\alpha)\in\Gamma\times A}\convmc 0$. In view of $\mathbb{mc}$-continuity of the addition in $X$,
this implies $(r_\gamma x_\alpha)_{(\gamma,\alpha)\in\Gamma\times A}\convmc rx$, which means that the scalar multiplication 
in $X$ is also $\mathbb{mc}$-continuous.

Thus, we proved that the convergence $\mathbb{mc}$ on $X$ is linear. Now, it is clear that the linear convergence $\mathbb{mc}$ satisfies the condition $(b)$ of 
Definition \ref{top+full+lattice convergence}, and hence $\mathbb{mc}$ is full. Proposition \ref{lattice in full} implies that $\mathbb{mc}$ is also a lattice convergence.
\end{proof}

In Example \ref{mw convergence}, $\mathbb{mw}\subseteq\mathbb{w}$. The inclusion is proper since 
$\mathbb{mw}$ is a lattice convergence by Theorem \ref{full lattice convergence to mc}, while $\mathbb{w}$ is not.

\begin{proposition}\label{mc on commutative d-algebras}
Let $\mathbb{c}$ be a linear full lattice convergence on a commutative $d$-algebra $X$.
For every net $x_\alpha$ in $X$, the following conditions are equivalent$:$\\
$(i)$ \ \ $x_\alpha\convmc 0$$;$\\
$(ii)$ \  $u\cdot x_\alpha\convc 0$ for every $u\in X_+$.
\end{proposition}

\begin{proof}
$(i)\Longrightarrow(ii)$\ Let $x_\alpha\convmc 0$ and $u\in X_+$. By Theorem \ref{full lattice convergence to mc},
$x_\alpha^{\pm}\convmc 0$ and so $u\cdot x_\alpha^{\pm}\convc 0$. Thus, $u\cdot x_\alpha=u\cdot x_\alpha^+-u\cdot x_\alpha^-\convc 0$.

$(ii)\Longrightarrow(i)$\ Let $u\in X_+$. By $(ii)$, it follows that  $|u\cdot x_\alpha|\convc 0$ because $\mathbb{c}$ is a lattice convergence. 
Due to the fact that $X$ is a $d$-algebra, $u\cdot|x_\alpha|=|u\cdot x_\alpha|\convc 0$. Since $u\in X_+$ is arbitrary, $x_\alpha\convmc 0$.
\end{proof}

\begin{lemma}\label{if c is full then multiplication is mc-continuous}
Let $\mathbb{c}$ be a linear full convergence on a commutative $l$-algebra $X$. Then the algebra multiplication in $X$ is $\mathbb{mc}$-continuous.
\end{lemma}

\begin{proof}
Let $x_\alpha\convmc x$ and $z\in X$. Since $\mathbb{c}$ is full and
$$
  u\cdot|z\cdot x_\alpha-z\cdot x|=u\cdot|z\cdot(x_\alpha-x)|\le u\cdot|z|\cdot|x_\alpha-x|\convc 0 \ \ \ \ (\forall u\in X_+), 
$$
it follows that  $z\cdot x_\alpha-z\cdot x\convmc 0$. Since $\mathbb{mc}$ is a linear convergence in view of Theorem \ref{full lattice convergence to mc},
we obtain $z\cdot x_\alpha\convmc z\cdot x$, as required.
\end{proof}

\begin{theorem}\label{product mc-convergence}
Let $\mathbb{c}$ be a linear full convergence on a commutative $l$-algebra $X$. If $(x_\alpha)_{\alpha\in A}\convmc x$, $(y_\beta)_{\beta\in B}\convmc y$
in $X$, and at least one of two nets is eventually order bounded then $(x_\alpha\cdot y_\beta)_{(\alpha,\beta)\in A\times B}\convmc x\cdot y$.
\end{theorem}

\begin{proof}
By Theorem \ref{full lattice convergence to mc}, $\mathbb{mc}$ is a linear full lattice convergence on $X$. By Lemma \ref{if c is full then multiplication is mc-continuous},
the algebra multiplication in $X$ is $\mathbb{mc}$-continuous. The rest of the proof follows now from Lemma \ref{product convergence}.
\end{proof}

\begin{theorem}\label{when mc=mmc}
Let $\mathbb{c}$ be a linear full lattice convergence on a commutative $l$-algebra $X$ with a positive algebraic unit.
Then $\mathbb{mc}$-convergence coincides with $\mathbb{mmc}$-conver\-gence. 
\end{theorem}

\begin{proof} By Lemma \ref{if c is full then multiplication is mc-continuous}, the multiplication in $X$ is $\mathbb{mc}$-continuous.
Now, Theorem \ref{c=mc} implies $\mathbb{mmc}=\mathbb{mc}$.
\end{proof}

By Example \ref{mw convergence}, there exists a full yet not lattice convergence $\mathbb{c}$ on a commutative $f$-algebra still
satisfying $\mathbb{mmc}=\mathbb{mc}$.

In the end of this section, we give the following condition on a $\mathbb{mc}$-convergence to be $T_1$ in a certain class of Archimedean $f$-algebras.   

\begin{theorem}\label{T_1 mc in f-algebras}
Let $\mathbb{c}$ be an additive $T_1$ convergence on an Archimedean $f$-algebra $X$.
Then $\mathbb{mc}\in T_1$ iff $X$ is semiprime.
\end{theorem}

\begin{proof}
Let $X$ be semiprime. If $\mathbb{mc}\not\in T_1$ then by Lemma \ref{T_1 for linear conv} $x\convmc y$ for some $x\ne y$ in $X$.
Hence $u\cdot|x-y|=|u\cdot(x-y)|\convc 0$ for every $u\in X_+$. Since $\mathbb{c}\in T_1$, then $|u\cdot(x-y)|=0$ and hence $u\cdot(x-y)=0$ for every $u\in X_+$. 
Then $w\cdot(x-y)=0$ for all $w\in X$. In particular, $(x-y)\cdot(x-y)=0$, violating the assumption that  $X$ is semiprime.
The obtained contradiction proves that $\mathbb{mc}\in T_1$.

Let $\mathbb{mc}\in T_1$. Suppose that $x\cdot x=0$ for some $x\in X$. Hence $|x|\cdot |x|=|x\cdot x|=0$.
Then $u\cdot|x|=0$ for all $u\in X_+$ by \cite[Prop. 10.2]{Pag}. By the definition of $\mathbb{mc}$-convergence, $x\convmc 0$.
Since $\mathbb{mc}\in T_1$, then  $x=0$ by Definition \ref{convergence}$(d)$. Therefore, $X$ is semiprime.
\end{proof}

\begin{corollary}\label{T_1 mo in f-algebras}
The $\mathbb{mo}$-convergence on an Archimedean $f$-algebra $X$ is a $T_1$-convergence iff $X$ is semiprime.
\end{corollary}

\begin{proof} 
It follows from Theorem \ref{T_1 mc in f-algebras} since $\mathbb{o}$ is an additive $T_1$ convergence.
\end{proof}

\section{$\mathbb{o}$-Convergence on Commutative $l$-Al\-ge\-bras}

We begin with the following useful characterization of $\mathbb{o}$-continuity of the algebra multiplication in an
$l$-algebra $X$, which  follows from the fact that multiplication by a fixed $y\in X_+$  is a positive operator on $X$.

\begin{lemma}\label{about o-cont multiplication}
The multiplication in a commutative $l$-algebra $X$ is $\mathbb{o}$-conti\-nuous iff 
$$ 
  z_\gamma\downarrow 0 \ \Rightarrow \  y\cdot z_\gamma\downarrow 0 \ \ \ \ \ (\forall y\in X_+).
$$
\end{lemma}

\begin{theorem}\label{proposition 3a}
Let $X$ be a commutative $l$-algebra. The following conditions are equivalent$:$\\
$(i)$ \ \ $X$ is a $d$-algebra with $\mathbb{o}$-continuous multiplication$;$\\
$(ii)$ \ $\inf u\cdot A=u\cdot\inf A$ for every $u\in X_+$ and $A\subseteq X$ such that $\inf A$ exists in $X$.
\end{theorem}

\begin{proof}
$(i)\Longrightarrow(ii)$\  
Let $X$ be a $d$-algebra i.e. 
$$
  u\cdot(x\wedge y)=(u\cdot x)\wedge(u\cdot y) \ \ \ (\forall u,x,y\in X_+),
  \eqno(16)
$$
and the multiplication in $X$ is $\mathbb{o}$-continuous. Let $A\subseteq X$ satisfy $\inf A\in X$. Without loss of generality, 
we may suppose $\inf A=0$. Consider the following downward directed set in $X$ 
$$
  A^\wedge=\left\{\bigwedge\limits_{k=1}^n a_k: a_k\in A\right\}.
$$
Clearly $A^\wedge\downarrow 0$. Pick any $u\in X_+$. $\mathbb{o}$-Continuity of the multiplication implies $u\cdot A^\wedge\downarrow 0$. By (16),
$$
   (u\cdot A)^\wedge=\left\{\bigwedge\limits_{k=1}^n (u\cdot a_k): a_k\in A\right\}= 
   \left\{u\cdot\bigwedge\limits_{k=1}^n a_k: a_k\in A\right\}=u\cdot A^\wedge,
$$
and hence $(u\cdot A)^\wedge\downarrow 0$, that means
$$
   \inf u\cdot A=\inf(u\cdot A)^\wedge=0=u\cdot 0=u\cdot\inf A,
$$	
as required.

$(ii)\Longrightarrow(i)$\   
Applying $(ii)$ to $A=\{x,y\}\subseteq X_+$ and $u\in X_+$ provides that $X$ is a $d$-algebra.
Let $(a_\xi)_{\xi\in\Xi}$ satisfy $a_\xi\downarrow 0$ in $X$. Then, for every $u\in X_+$, $u\cdot a_\xi\downarrow$. By $(ii)$,
$$
  \inf\{u\cdot a_\xi:\xi\in\Xi\}=\inf u\cdot\{a_\xi:\xi\in\Xi\}=u\cdot\inf\{a_\xi:\xi\in\Xi\}=u\cdot 0=0.
$$
Thus $u\cdot a_\xi\downarrow 0$, which implies $\mathbb{o}$-continuity of the algebra multiplication by Lemma \ref{about o-cont multiplication}. 
\end{proof}
\noindent

\begin{remark} 
\begin{enumerate}
\item[$(i)$] \ Every Archimedean $f$-algebra satisfies the conditions of Theorem $\ref{proposition 3a}$ due to
$\mathbb{o}$-continuity of its algebra multiplication.
\item[$(ii)$] \ A commutative Archimedean $d$-algebra does not satisfy  the condition $(ii)$ of  Theorem $\ref{proposition 3a}$
in general $($see, e.g., Example $\ref{almost f-algebra in which multiplication is not 0-cont}$$)$.
\item[$(iii)$] \ The condition $(ii)$ of Theorem $\ref{proposition 3a}$ can be considered as a generalization of the notion of a $d$-algebra 
$($cf. Definition $\ref{various lattice algebras}$$(a)$$)$ and it was referred to as the {\em infinite distributive property} in $\cite{Ay1}$.
\end{enumerate}
\end{remark}

\begin{proposition}\label{mo net of nilpotents}
Let $X$ be an Archimedean $f$-algebra. Then the set $N(X)$ of all nilpotent elements in $X$ is $\mathbb{mo}$-closed.
\end{proposition}

\begin{proof}
The result follows from Lemma \ref{mc net of nilpotents}, since $\mathbb{o}$ is an additive $T_1$ convergence on $X$
by Proposition \ref{o-convergence is full lattice and T_1}.
\end{proof}

\begin{example}\label{sequential o} 
Consider the $l$-algebra $s_{\omega}(T)$ of all countably supported real functions on an uncountable set $T$  with the poitwise multiplication.
It is a straightforward to check that$:$
\begin{enumerate} 
\item[$(a)$] \ $s_{\omega}(T)$ is universally $\sigma$-complete semiprime non-unital $f$-algebra$;$
\item[$(b)$] \ the $\mathbb{o}$-convergence on the Dedekind complete Riesz space $s_{\omega}(T)$ is sequential in the sense of Definition $\ref{sequential c-convergence}$$;$
\item[$(c)$] \ the $\mathbb{o}$-convergence is not topological in $s_{\omega}(T)$ due to $\cite[Thm.1]{DEM3}$.
\end{enumerate} 
\end{example}

Although it is beyond the scope of the present paper, it is worth mentioning that many interesting applications of sequential 
full convergences are related to Koml{\'o}s-like properties in Riesz spaces and in $l$-algebras (see, e.g., \cite{GTX,EEG}).

The following well known fact will be used later.  Since we did not find an appropriate reference to this proposition in the literature,
we include its proof.

\begin{proposition}\label{extension of multiplication to Dedekind completion}
Let $X$ be an Archimedean commutative $l$-algebra $X$ with  $\mathbb{o}$-continuous algebra multiplication $*$.
Then $*$ admits a unique extension $*$ to the Dedekind completion $X^\delta$ of $X$ which makes $X^\delta$ an $l$-algebra
with  $\mathbb{o}$-continuous algebra multiplication $*$. If $X$ is also an $f$-algebra, almost $f$-algebra, or $d$-algebra, 
the same is true for $X^\delta$ equipped with the extended algebra multiplication $*$.   
Furthermore, if $X$ has  multiplicative unit $e$, then $e$ is also the multiplicative unit in $X^\delta$.
\end{proposition}

\begin{proof}
First, arguing as in \cite[p.67]{Pag}, we extend uniquely $*$ to the algebra multiplication $*$ in $X^\delta_+$.
Given $x^\delta,y^\delta\in X^\delta_+$, there exist $x,y\in X_+$ such that
$$ 
   x^\delta=\sup\limits_{X^\delta}\{w\in X_+|w\le x^\delta\}\le x \ \ \text{and} \ \ y^\delta=\sup_{X^\delta}\{w\in X_+|w\le y^\delta\}\le y.
$$
Since $0\le u*v\le x*y$ for all $u,v\in X$ with $0\le u\le x^\delta$ and $0\le v\le y^\delta$, there exists
$$
    z^\delta=\sup_{X^\delta}\{u*v: u,v\in X_+ ; u\le x^\delta; v\le y^\delta\}.
$$
We define $x^\delta * y^\delta:=z^\delta\in X^\delta$. It should be clear that 
the so defined multiplication makes $X^\delta$ an $l$-algebra and extends the original multiplication in $X$ to $X^\delta$.
It is routine to prove that the extended multiplication satisfies conditions $(a)$, $(b)$, $(c)$, or $(e)$ of Definition \ref{various lattice algebras} 
in $X^\delta$ whenever the original multiplication satisfies the corresponding conditions in $X$.

In order to show that the extended multiplication is $\mathbb{o}$-continuous, suppose $z^{\delta}_\xi\downarrow 0$ in $X^\delta$ and $y\in X_+$. 
Clearly $U:=\{x\in X|(\exists\xi) x\ge z^{\delta}_\xi\}$ satisfies $U\downarrow 0$ in $X$. By the assumption, $y*U\downarrow 0$. 
Therefore, $y*z^{\delta}_\xi\downarrow 0$ in $X^\delta$.  Now let $z^{\delta}_\xi\downarrow 0$ in $X^\delta$ and $y\in X^\delta_+$. 
Since $y^\delta=\sup_{X^\delta}\{w\in X_+|w\le y^\delta\}\le y$ for some $y\in X_+$, we get    
$$
  0\le y^\delta*z^{\delta}_\xi\le y*z^{\delta}_\xi\downarrow 0, 
$$
and hence $y^\delta*z^{\delta}_\xi\downarrow 0$ in $X^\delta$. By Lemma \ref{about o-cont multiplication},
the extended multiplication $*$ is $\mathbb{o}$-continuous in  $X^\delta$.

Suppose that $\circ$ is another $\mathbb{o}$-continuous algebra multiplication in $X^\delta$ that makes $X^\delta$ an $l$-algebra and extends $*$ from $X$. 
Let $\hat{x},\hat{y}\in X^\delta$ and take nets $x_\alpha$ and $y_\beta$ in $X$ such that $x_\alpha\convo\hat{x}$ and $y_\beta\convo\hat{y}$ in $X^\delta$. 
Since any $\mathbb{o}$-Cauchy net in $X$ is  eventually order bounded in $X$, the nets $x_\alpha$ and $y_\beta$ are both eventually order bounded.
Due to  $\mathbb{o}$-continuity of the both extended multiplications, Lemma \ref{product convergence} implies
$x_\alpha*y_\beta\convo\hat{x}\cdot\hat{y}$ and $x_\alpha*y_\beta\convo\hat{x}\circ\hat{y}$.
Since $\hat{x},\hat{y}\in X^\delta$ were chosen arbitrary, this ensures uniqueness of the extension.
\end{proof}

It is worth to mention that $X^\delta$ can be unital even when $X$ is an $f$-algebra without multiplicative unit \cite[Ex.10.10]{Pag}.\\ 

In what follows, $X$ is an Archimedean Riesz space. We remind that a  Dedekind complete Riesz space is called {\em universally complete} if each non-empty pairwise 
disjoint subset of its positive elements has a supremum. It is a straightforward consequence of the Zorn lemma that each universally complete Riesz space has a weak unit. 
It is well known that any Archimedean Riesz space $X$ has a Dedekind and a universal completion, unique up to lattice isomorphism which are denoted by $X^\delta$  
and $X^u$. We always suppose $X\subseteq X^\delta\subseteq X^u$, whereas $X^\delta$ is an order ideal in $X^u$. Since $X$ is order dense in $X^\delta$ and $X^\delta$ is order 
dense in $X^u$, $X$ is regular in both $X^\delta$ and $X^u$. It is also well known that $X$ is {\it discrete} iff $X$ is lattice isomorphic to an order dense Riesz subspace 
of the universally complete Riesz space ${\mathbb{R}^D}$ for some $D$ (cf. \cite[Theorem 1.78]{AB}).

The classical result of Gordon \cite[Theorem 2]{Go1} (cf. also \cite[Thms 8.1.2 and 8.1.6]{Ku}) expresses the immanent relations between 
Riesz spaces and Boole\-an-valued analysis. It can be stated briefly as follows: {\em every universally complete Riesz space is an interpretation 
of the reals ${\mathcal R}$ in an appropriate Boolean-valued model $V^{(B)}$}.  For  terminology and elementary techniques of Boolean-valued analysis
not explained in this paper we refer  the reader to \cite{Go1,Go2,Go3,Ku,GKK,EGK}.

\begin{theorem}[The Gordon theorem]\label{Gordon}
Let $X$ be an Archimedean Riesz space with the $($always complete$)$ Boolean algebra $B={\mathfrak{B}(X)}$ of all bands in $X$ and ${\mathcal R}$ the real field 
in $V^{(B)}$. The {\em descent} ${\mathcal R}\downarrow:=\{x\in V^{(B)}: [\![x\in{\mathcal R}]\!]=1_B\}$ of ${\mathcal R}$
is a universally complete Riesz space including $X$ as an order dense Riesz subspace. Moreover,
$$
  bx\le by \ \Longleftrightarrow\ b\le [\![x \le y]\!] \ \ \ \ \ \ (\forall b\in B)(\forall x,y\in{\mathcal R}\downarrow),
$$
where $B$ is identified with the Boolean algebra $\mathfrak{P}(X^\delta)=\mathfrak{P}(X^u)$ of all band projections in $X^\delta$,
and therefore in $X^u$ since $X^\delta$ is an order ideal in $X^u$. 
\end{theorem}

By the Gordon theorem, the universal completion $X^u$ of an Archi\-me\-dean Riesz space $X$ is the descent ${\mathcal R}\downarrow$
of the reals ${\mathcal R}$ in $V^{(\mathfrak{B}(X))}$, and  uniqueness of $X^u$ up to an order isomorphism follows
from  uniqueness of ${\mathcal R}$ in $V^{(\mathfrak{B}(X))}$ (cf. \cite[8.1.7.]{Ku}). Clearly any $x\in {\mathcal R}\downarrow$
such that $[\![x>0]\!]=1_B$ is a weak order unit in ${\mathcal R}\downarrow$. Furthermore, ${\mathcal R}\downarrow$ is a unital $f$-algebra furnished 
with the following $f$-algebra multiplication: 
$$
    x\cdot y =z \ \ \text{whenever} \ \  [\![x\cdot y=z]\!]=1_B \ \ \ \ (\forall x,y,z\in {\mathcal R})
$$
(see \cite[Thm.1.6]{Go2}, \cite[8.1.3]{Ku}, and \cite[A3.2]{GKK}). In what follows, we denote this 
{\em natural representational $f$-algebra multiplication} in $X^u$ by $\cdot$. Moreover, any other semiprime 
$f$-algebra multiplication $\circ$ in $X^u$ is uniquely determined by fixing a multiplicative unit (cf. \cite[1.4.6(3)]{Ku}).

The rest of the section is devoted to applications of the Gordon theorem to $\mathbb{mo}$-convergence on $f$-algebras. We include
the following well known result (see \cite{Ku,BC} and references therein) and provide its straightforward Boolean-valued proof 
based on Theorem \ref{Gordon}.

\begin{proposition}\label{universally complete f-algebra}
Let $X$ be an Archimedean $f$-algebra with the algebra multiplication $*$ and let $X^u$ be an $f$-algebra with its natural  
algebra multiplication $\cdot$ and the multiplicative unit $e$. Then there exists $\zeta\in X^u_+$ such that
$x*y=\zeta\cdot x\cdot y$ for all $x,y\in X$. 

Furthermore, $X$ is $*$-semiprime iff $\zeta$ is a weak order unit in $X^u$.
In this case, $*$ has the unique semiprime extension to $X^u$ obtaining by the formula $x*y:=\zeta\cdot x\cdot y$ for all $x,y\in X^u$.
\end{proposition}

\begin{proof}
The idea of the proof can be outlined as follows:
in the first step we use that the $f$-al\-ge\-b\-ra multiplication $*$ in an Ar\-chi\-medean Riesz space $X$ extends 
to an $f$-algebra multiplication in $X^u$ and produces an $f$-algebra structure on the real field
${\mathcal R}$ in an appropriate Boolean-valued model $V^{(B)}$; in the second step we use the Boolean-valued interpretation of the fact that 
a non-trivial $f$-algebra structure on $\mathbb{R}$ is uniquely determined by fixing of its multiplicative unit 
(cf. Example \ref{l-algebra $R$}).

Take a maximal pairwise disjoint family $\{u_\alpha\}_{\alpha\in A}$ in $X_+$, and let ${\mathfrak{B}(B_{u_\alpha})}$ be the
Boolean algebra of all bands in the band $B_{u_\alpha}$ generated by $u_\alpha$ in $X^\delta$.
By Proposition \ref{extension of multiplication to Dedekind completion}, the $f$-algebra multiplication $(*)$ possesses the unique 
extension $(*)$ to $X^\delta$ and hence to each of $B_{u_\alpha}$.  
Clearly $(*)$ extends further to the $f$-algebra multiplication: 
$$ 
  (a_\alpha)_{\alpha\in A} * (b_\alpha)_{\alpha\in A}:=(a_\alpha*b_\alpha)_{\alpha\in A}
$$
 in the Dedekind complete Riesz space $\hat{X}=\prod_{\alpha\in A}B_{u_\alpha}$. 

Let $B:={\mathfrak{B}(X)}$. Clearly ${\mathfrak{B}(X)}={\mathfrak{B}}(\hat{X})={\mathfrak{B}}(X^u)$.
By the Gordon theorem, $u:=\sup_{\hat{X}}u_\alpha\in\hat{X}\subseteq{\mathcal R}\downarrow$ in $V^{(B)}$.
Since $u$ is a weak order unit in ${\mathcal R}\downarrow=X^u$, then $[\![u>0]\!]=1_B$. Take 
$w\in {\mathcal R}\downarrow$ such that $[\![ u\cdot w=1_{{\mathcal R}}]\!]=1_B$ and extend 
the $f$-algebra multiplication $(*)$ from $\hat{X}$ to ${\mathcal R}\downarrow$ by letting
$$
    a*b=(a\cdot u\cdot w)*(b\cdot u\cdot w):=a\cdot b\cdot w\cdot w\cdot (u*u) \ \ \ \ \ (\forall a,b\in {\mathcal R}\downarrow).
$$
Clearly 
$$
  [\![\ * \ \text{and} \ \cdot \ \text{are\ both \  f-algebra \ multiplications \ in} \ {\mathcal R} \ ]\!]=1_B.
$$
Denote $\zeta:=1_{{\mathcal R}}*1_{{\mathcal R}}$. Then, for all $a,b\in {\mathcal R}\downarrow$, 
$$
  [\![a*b=(a\cdot 1_{{\mathcal R}})*(b\cdot 1_{{\mathcal R}})=a\cdot b\cdot(1_{{\mathcal R}}*1_{{\mathcal R}})=\zeta\cdot a\cdot b]\!]=1_B,
$$
and hence $a*b=\zeta\cdot a\cdot b$ as required. 

Let $X$ be $*$-semiprime. Take $x\in X_+$ with $x\wedge\zeta=0$. Then $x*x=\zeta\cdot x\cdot x=0$ and, since $X$ is  $*$-semiprime, 
$x=0$ shows that $\zeta$ is a weak unit in $X^u$.

Let $\zeta$ be a weak order unit in $X^u$. Take $x\in X$ with $x*x=0$. Then $\zeta\cdot x\cdot x=0$ and hence $x\cdot x=0$.
As it was remarked above, $X^u$ is $\cdot$ -semiprime. Therefore $x=0$ which implies that $X$ is $*$-semiprime.
\end{proof}

\begin{theorem}\label{mo-conv in X is equiv to o-conv in X^u}
Let $x_\alpha$ be a net in an Archimedean semiprime $f$-algebra $X$ with algebra multiplication $*$ and weak order unit $w$.
Let $X^u$ be equipped with the unique semiprime $f$-algebra extension $*$ of the multiplication in $X$.
The following conditions are equivalent$:$
\begin{enumerate}
\item[$(i)$] \ $x_\alpha\convmo 0$ in $X$$;$
\item[$(ii)$] \ $x_\alpha\convo 0$ in $X^u$$;$
\item[$(iii)$] \ $x_\alpha\convmo 0$ in $X^u$.
\end{enumerate}
\end{theorem}

\begin{proof}
We apply Proposition \ref{universally complete f-algebra} to the standard semiprime $f$-algebra multiplication $\cdot$ in $X^u$ and the  
extended $f$-algebra multiplication $*$ in $X^u$ given by $a*b=\zeta\cdot a\cdot b$ for all $a,b\in X^u={\mathcal R}\downarrow$ and an appropriate 
weak unit $\zeta\in X^u={\mathcal R}\downarrow$.

$(i)\Longrightarrow(ii)$ Let $x_\alpha\convmo 0$ in $X$. Then $w*x_\alpha\convo 0$ in $X$ and hence in $X^u$.
Since $X$ is regular in $X^u={\mathcal R}\downarrow$, $w*x_\alpha\convo 0$ in ${\mathcal R}\downarrow$,
and hence $\zeta\cdot w\cdot x_\alpha\convo 0$. Taking $\varsigma\in{\mathcal R}\downarrow$ such that $\varsigma\cdot\zeta\cdot w=1_{{\mathcal R}}$ 
and using $\mathbb{o}$-continuity of the algebra multiplication $\cdot$ in $X^u$, we obtain
$x_\alpha=(\varsigma\cdot\zeta\cdot w)\cdot x_\alpha=\varsigma\cdot(\zeta\cdot w\cdot x_\alpha)\convo 0$ in $X^u$.

$(ii)\Longrightarrow(iii)$ follows from Proposition \ref{o-convergence is full lattice and T_1} and Theorem \ref{c=mc} 
since $X^u$ is unital and the multiplication in $X^u$ is $\mathbb{o}$-continuous.

$(iii)\Longrightarrow(i)$ is trivial since $X$ is an $f$-subalgebra of $X^u$.  
\end{proof}

\begin{remark}\label{related}
\begin{enumerate}
In circumstances of Theorem $\ref{mo-conv in X is equiv to o-conv in X^u}$$:$
\item[$(i)$] \ We do not know whether or not the statement of  Theorem $\ref{mo-conv in X is equiv to o-conv in X^u}$ remains true without the assumption that $X$ has a weak order unit.
\item[$(ii)$] \ Theorem $\ref{mo-conv in X is equiv to o-conv in X^u}$ implies $\mathbb{mo}=\mathbb{o}\subseteq\mathbb{uo}$ in $X^u$.
\item[$(iii)$] \ By Theorem $4$ of $\cite{EGK}$, Theorem $\ref{mo-conv in X is equiv to o-conv in X^u}$  implies that $\mathbb{mo}=\mathbb{uo}$ in $X^u$ iff $\dim(X)<\infty$.
\item[$(iv)$] \ It follows from Theorem $\ref{mo-conv in X is equiv to o-conv in X^u}$ that $\mathbb{umo}=\mathbb{uo}$ in $X^u$.
\item[$(v)$] \ In view of Theorem $\ref{mo-conv in X is equiv to o-conv in X^u}$,  Theorem $4$ of $\cite{EGK}$  implies that  $\mathbb{umo}=\mathbb{mo}$ in $X^u$ iff $\dim(X)<\infty$.
\end{enumerate}
\end{remark}

\begin{example}\label{almost f-algebra mult is not extendable to X^u in general}
Consider the $l$-algebra $X=(\ell_\infty,*)$ from Example $\ref{almost f-algebra in which multiplication is not 0-cont}$. 
There is no $l$-algebra multiplication extension of $*$ to $X^u=s$. To see this take $x:=(k)_{k=1}^\infty\in s$.
Then, in the notations of Example $\ref{almost f-algebra in which multiplication is not 0-cont}$, 
$$
  \mathbb{1}*x\ge\mathbb{1}*n\mathbb{1}_n=n\mathbb{1} \ \ \ \ (\forall n\in\mathbb{N}),
$$
which shows that $\mathbb{1}*x$ cannot be defined in the Riesz space $s$.
\end{example}

The following lemma could be known. As we did not find it in an available literature, we also include its proof.   

\begin{lemma}\label{B_u=B_u^2}
Let $X$ be a $\sigma$-Dedekind complete $f$-algebra. The following conditions are equivalent$:$\\
$(i)$ \ $X$ is semiprime$;$\\
$(ii)$ \ $B_u=B_{u^2}$ for every $u\in X_+$, where $B_u$ is a band generated by $u$.
\end{lemma}

\begin{proof}
$(ii)\Longrightarrow(i)$ is trivial.
  
$(i)\Longrightarrow(ii)$ Clearly $B_{u^2}\subseteq B_u$.  Let $0<z\in B_u\setminus B_{u^2}$. Since $X$ has the principle projection property,
$pr_{u^2}z$ exists. Since $z\not\in B_{u^2}$ then $v:=z-pr_{u^2}z>0$. Clearly $v\in  B_u\cap B_{u^2}^d$.
Since $v\in B_u$, there exists a net $w_{\alpha}\in I_u$ with $0\le w_{\alpha}\uparrow v$. Since the algebra multiplication in any $f$-algebra is 
$\mathbb{o}$-continuous and since $\mathbb{o}$-convergence on $X$ is a linear full lattice convergence due to Proposition \ref{o-convergence is full lattice and T_1}, 
Lemma \ref{product convergence} implies $w_{\alpha}^2\uparrow v^2$. Clearly $w_{\alpha}^2\in I_{u^2}$ for any $\alpha$.
Therefore $v^2\in B_{u^2}$. Since $v\in  B_u\cap B_{u^2}^d$ then $v\wedge u^2=0$ and since $X$ is an $f$-algebra, we obtain $v^2\wedge u^2=0$
or $v^2\in B_{u^2}^d$. Combining this with $v^2\in B_{u^2}$ gives  $v^2=0$ on the contrary with $v>0$ as $X$ is semiprime.
\end{proof}

\begin{theorem}\label{cont of uo-multiplication}
The algebra multiplication is $\mathbb{uo}$-continuous in any $\sigma$-Dedekind complete semiprime $f$-algebra $X$.
\end{theorem}

\begin{proof}
Let $u\in X_+$ and $x_{\alpha}\convuo 0$ in $X$. By Lemma \ref{B_u=B_u^2}, $u\cdot|x_{\alpha}|\in B_u=B_{u^2}$ for all $\alpha$. 
Since the algebra multiplication in $X$ is $\mathbb{o}$-continuous, 
$$
    u\cdot v\wedge |u\cdot x_{\alpha}|=u\cdot v\wedge u\cdot|x_{\alpha}|=u\cdot(v\wedge|x_{\alpha}|)\convo 0 \ \ \ (\forall v\in X_+).
$$
In particular, $u^2\wedge |u\cdot x_{\alpha}|\convo 0$ in $B_{u^2}=B_u$. It follows from \cite[Cor.3.5]{GTX} that $u\cdot x_{\alpha}\convuo 0$ in $B_{u^2}=B_u$. 
Since $B_u$ is regular in $X$, \cite[Thm.3.2]{GTX} implies $u\cdot x_{\alpha}\convuo 0$ in $X$, as required.
\end{proof}

\begin{corollary}\label{cont of uo-multiplication in o-complete}
The algebra multiplication is $\mathbb{uo}$-continuous in any Dedekind complete $f$-algebra $X$.
\end{corollary}

\begin{proof}
It follows by applying Theorem \ref{cont of uo-multiplication} to the band decomposition $X=N(X)+N(X)^d$.
\end{proof}

\begin{theorem}\label{multiplication in any universally complete f-algebra is uo-cont}
The algebra multiplication in any universally complete $f$-algebra $X$ is $\mathbb{uo}$-continuous.
\end{theorem}

\begin{proof}
Form the band decomposition $X=N(X)+N(X)^d$ and apply Proposition \ref{universally complete f-algebra} and then Theorem \ref{cont of uo-multiplication}
to its second part. In the first part of the band decomposition the conclusion is trivial. 
\end{proof}

\section{${\text{\boldmath{$m\tau$}}}$-Convergence on Commutative $l$-Algebras}

In the end of the paper we present the following description of the ${\text{\boldmath{$m\tau$}}}$-convergence on locally full commutative $l$-algebras. 

\begin{theorem}\label{mtau-topology}
Let $X$ be a commutative $l$-algebra $X$ with a linear locally full topology $\tau$. Then the ${\text{\boldmath{$m\tau$}}}$-convergence on $X$ 
is topological with respect to a linear locally solid topology $\tau_m$ on $X$. 
\begin{enumerate}
\item[$(i)$] \ If $X$ has a positive algebra unit $e$, and $\tau$ is Hausdorff, then the $\tau_m$-topology is Hausdorff. 
\item[$(ii)$] \ If the multiplication in $X$ is ${\text{\boldmath{$\tau$}}}$-continuous, and $\tau$ is locally solid, then $\tau_m\subseteq\tau$. 
\end{enumerate}
\end{theorem}

\begin{proof}
Let $\mathcal{T}$ be the collection of all full $\tau$-neighborhoods of zero, 
$$
  U_{a,A}:=\{x\in X: a\cdot|x|\in A\}  \  \  \  \  (a\in X_+,  A\in\mathcal{T}),
  \eqno(17)
$$
and $\mathcal{B}:=\{U_{a,A}: a\in X_+, A\in\mathcal{T}\}$. We shall show that $\mathcal{B}$ is a zero base 
of some linear locally solid topology $\tau_m$ on $X$. For this purpose we need to show that: 
\begin{enumerate}
\item[$a)$] \ if $V_1,V_2\in\mathcal{B}$ then there exists $W\in \mathcal{B}$ with $W\subseteq V_1\cap V_2$;  
\item[$b)$] \ $\mathcal{B}$ consists of absorbing sets;
\item[$c)$] \ $\mathcal{B}$ consists of balanced sets;
\item[$d)$] \ For each $V\in \mathcal{B}$ there exists some $W\in \mathcal{B}$ with $W+W \subseteq V$;
\item[$e)$] \ $\mathcal{B}$ consists  of solid sets.
\end{enumerate}
For proving $a)$ take $U_{a,A},U_{b,B}\in\mathcal{B}$.  Since $(a\vee b)\cdot|x|\in A\cap B\subseteq A$ for every $x\in U_{a\vee b, A\cap B}$, 
then $a\cdot|x|\in A$, because $A$ is full and $0\le a\cdot|x|\le (a\vee b)\cdot|x|$. Thus $\mathcal{B}\ni U_{a\vee b, A\cap B}\subseteq U_{a, A}\cap U_{b, B}$. 
Condition $b)$ is an immediate consequence of (17) since $\mathcal{T}$ consists of absorbing sets. For proving $c)$ let $t\in \mathbb{R}$, $|t| \le 1$,
and $x \in U_{a,A}$. Then $0\le a\cdot|t x| \le a\cdot |x| \in A$ and hence $tx \in U_{a,A}$ since $A$ is a full neighborhood of zero. Thus we have shown 
every  $U_{a,A}\in\mathcal{B}$ is balanced. For proving $d)$ take $U_{a,A}\in\mathcal{B}$. Since $\mathcal{T}$ is a zero base for topology $\tau$,
there exists $A'\in\mathcal{T}$ with $A'  + A' \subseteq A$. If $x, y \in U_{a,A'}$ then $a\cdot|x|,a\cdot|y|\in A'$. It follows from
$a\cdot|x+y| \le a\cdot|x| + a\cdot|y| \in A' + A' \subseteq A$ that $x+y \in A$. Thus, we proved that $U_{a,A'} + U_{a,A'}  \subseteq U_{a,A}$.
For proving $e)$ take $x\in U_{a,A}$ and let $|y|\le|x|$. Then $|a\cdot|y||=a\cdot|y|\le a\cdot|x|\in A$ implies $a\cdot|y|\in A$ since $A$ is full, 
and hence $y\in U_{a,A}$. 

We have shown that: $\mathcal{B}$ is a zero base of a linear topology on $X$, which will be called by $\tau_m$;
$\mathcal{B}$ consists of solid (but not necessary full) $\tau_m$-neighborhoods of $0$.

Clearly $x_\alpha\convmtau 0$ iff $a\cdot|x_\alpha|\convtau 0$  for every $a\in X_+$ iff 
for every $a\in X_+$ and every $A\in\mathcal{T}$ there exists $\alpha_0$ such that 
$a\cdot|x_\alpha|\in A$ for all $\alpha \ge \alpha_0$
iff $x_\alpha\in U_{a,A}\in\mathcal{B}$ for all $\alpha\ge\alpha_0$ iff $x_\alpha\convtaum 0$.  

$(i)$ \  Now let $e$ be a positive algebra unit in $X$ and let $\tau$ be Hausdorff. Since $x_\alpha\convmtau 0$ iff $x_\alpha\convtaum 0$, 
in order to prove  that $\tau_m$ is Hausdorff, it is enough to show that the ${\text{\boldmath{$m\tau$}}}$-conver\-gence is $T_1$.
By Lemma \ref{T_1 for linear conv}, we have to show that if $x\convmtau 0$ then $x=0$. So let $x\convmtau 0$.  Then 
$u\cdot|x|\convtau 0$ for every $u\in X_+$, in particular,
$|x|=e\cdot|x|\convtau 0$. This means $|x|=0$ and hence $x=0$, as required. 

$(ii)$ \ Now, let the  multiplication in $X$ be ${\text{\boldmath{$\tau$}}}$-continuous and $\tau$ be locally solid. 
Then $|\cdot|$ is also a ${\text{\boldmath{$\tau$}}}$-continuous operation in $X$.
Then $\tau_m$-neighborhoods from $(17)$ belong to $\tau$, and hence $\tau_m\subseteq\tau$. 
\end{proof}

We include the following example of a commutative non-unital locally solid almost $f$-algebra $(X,\tau)$ in which the algebra multiplication 
is not ${\text{\boldmath{$\tau$}}}$-continuous.  

\begin{example}\label{last}
Consider the $l$-algebra $X=(\ell_\infty,*)$ from Example $\ref{almost f-algebra in which multiplication is not 0-cont}$ 
with respect to the locally solid topology inherited from the Tychonoff topology $\tau$ on $X^u=s=\mathbb{R}^\mathbb{N}$.
Since $\mathbb{uo}$-convergence on $X$ coincides with the ${\text{\boldmath{$\tau$}}}$-convergence $($see, e.g., $\cite[Thm.2]{DEM3}$$)$ and since 
$*$ is not $\mathbb{uo}$-continuous by Example $\ref{almost f-algebra in which multiplication is not 0-cont}$, 
the multiplication $*$ in the locally solid $l$-algebra $(\ell_\infty,*,\tau)$ is not ${\text{\boldmath{$\tau$}}}$-continuous.
\end{example}

However, the algebra multiplication in the almost $f$-algebra $(X,\tau)$ from Example \ref{last} is ${\text{\boldmath{$m\tau$}}}$-continuous. 
More generally, let $(X,\tau)$ be a commutative locally full $l$-algebra. 
\begin{enumerate}
\item[$(i)$] \ Since the ${\text{\boldmath{$\tau$}}}$-convergence on $X$ is a linear full lattice convergence by Theorem $\ref{locally solid convergence}$, 
the ${\text{\boldmath{$m\tau$}}}$-convergence is a linear full lattice convergence by Theorem $\ref{full lattice convergence to mc}$. 
The ${\text{\boldmath{$m\tau$}}}$-continuity of the algebra multiplication in $X$ follows then from Lemma $\ref{if c is full then multiplication is mc-continuous}$.
\item[$(ii)$] \ If $X$ has a positive algebraic unit then, by Lemma $\ref{full topology}$, Lemma $\ref{mc implies c}$, and Theorem \ref{when mc=mmc},
${\text{\boldmath{$mm\tau$}}}={\text{\boldmath{$m\tau$}}}$ and the ${\text{\boldmath{$m\tau$}}}$-convergence implies ${\text{\boldmath{$\tau$}}}$-convergence.
If $X$ is non-unital then, in general, the ${\text{\boldmath{$m\tau$}}}$-convergence does not imply the ${\text{\boldmath{$\tau$}}}$-convergence, e.g. by Example \ref{c_00}.
\end{enumerate}

{\tiny 

}
\end{document}